\documentclass{amsart}
\usepackage{amssymb,amsmath,amsthm}
\usepackage[shortlabels]{enumitem}
\usepackage{comment}

\newtheorem{theorem}{Theorem}[section]

\newtheorem{lemma}[theorem]{Lemma}
\newtheorem{proposition}[theorem]{Proposition}
\newtheorem{definition}[theorem]{Definition}
\newtheorem{corollary}[theorem]{Corollary}

\theoremstyle{definition}
\newtheorem{remark}[theorem]{Remark}

\newcommand{\std}{\mathrm{std}}

\newcommand{\N}{\mathbb{N}}
\newcommand{\Z}{\mathbb{Z}}
\newcommand{\R}{\mathbb{R}}
\newcommand{\T}{\mathbb{T}}
\newcommand{\C}{\mathbb{C}}

\newcommand{\cU}{\mathcal{U}}

\newcommand{\cT}{\mathcal{T}}
\newcommand{\cS}{\mathcal{S}}
\newcommand{\cF}{\mathcal{F}}

\newcommand{\id}{\mathrm{id}}

\title[Selfless $C^*$-correspondences]{Selfless $C^*$-correspondences, operator valued $C^*$-probability spaces and completely positive maps}

\author[D. Gao]{David Gao}

\author[M. Junge]{Marius Junge}

\author[S. Kunnawalkam Elayavalli]{Srivatsav Kunnawalkam Elayavalli}
\address{\parbox{\linewidth}{Department of Mathematics, University of Maryland, College Park, \\
		4176 Campus Dr, College Park, MD 20742}}
\email{sriva@umd.edu}
\urladdr{https://sites.google.com/view/srivatsavke}

\author[G. Patchell]{Gregory Patchell}
\address{\parbox{\linewidth}{Mathematical Institute, University of Oxford, Andrew Wiles Building, \\ Radcliffe Observatory Quarter, Woodstock Road, Oxford, OX2 6GG, UK}}
\email{greg.patchell@maths.ox.ac.uk}
\urladdr{https://sites.google.com/view/gpatchel}

\author[L. Robert]{Leonel Robert}
\address{\parbox{\linewidth}{Department of Mathematics, University of Louisiana at Lafayette, \\
		217 Maxim Doucet Hall, 1401 Johnston Street, Lafayette, LA 70503, USA}}
\email{lrobert@louisiana.edu}

\begin{document}

\begin{abstract}
    We develop a general theory of selflessness for $C^*$-correspondences, with several applications. Specializing this theory to completely positive maps, in particular to conditional expectations, gives rise to a novel notion of relative selflessness. Among applications, the machinery developed here yields new examples of selfless $C^*$-algebras, for instance among minimal tensor products and reduced crossed products, new examples of MF $C^*$-algebras arising as reduced amalgamated free products, a conceptually new proof of Kirchberg's $\mathcal{O}_{\infty}$-absorption theorem, and the lack of ``phantom'' traces on ultrapowers.
\end{abstract}
\maketitle

\section{Introduction}
The notion of selfless $C^*$-probability space was introduced by the fifth author in \cite{robertselfless}. Recall that a C*-probability space $(A,\rho)$, where $A$ is a unital C*-algebra and $\rho$ is a  state on $A$ inducing a faithful GNS representation, is called selfless if the embedding 
$$
(A,\rho)\to (A,\rho)*(C,\kappa)
$$
is existential for  some $C^*$-probability space $(C,\kappa)$, with $C\neq \C$ and $\kappa$ a state inducing a faithful GNS representation of $C$. As demonstrated in \cite{robertselfless}, selfless $C^*$-algebras possess many desirable properties, including simplicity, having at most one trace,  and in the tracial case having stable rank 1 and the strict comparison property.

The work \cite{amrutam2025strictcomparisonreducedgroup} proved selflessness for a wide class of reduced group $C^*$-algebras, using an approach involving the rapid decay property.  Selflessness has since been established for broad classes of reduced group $C^*$-algebras  \cite{elayavalli2025remarks, ozawa2025proximalityselflessnessgroupcalgebras, yang2025extreme, BradfordSisto2026}, twisted group $C^*$-algebras \cite{raum2025strictcomparisontwistedgroup, flores2026purenessstablerankreduced},
reduced free products \cite{HKER, hayes2025selfless, fmmm2025selflessfreeprod}, graph products \cite{fmmm2025selflessfreeprod}, amalgamated free products and HNN extensions \cite{gao2026selflessreducedamalgamatedfree}, linear groups \cite{vigdorovich2026selflessreducedcalgebraslinear}, and compact quantum groups \cite{hayes2025selfless}.
Notably, in \cite{ozawa2025proximalityselflessnessgroupcalgebras} Ozawa  settled several natural questions concerning selflessness, developed a  streamlined approach to proving this property, defined a notion of complete selflessness, and introduced  the PHP property for discrete groups.

This article presents a new framework that expands the reach of selflessness to a very broad generality, with various applications. We begin by discussing the case of arbitrary completely positive maps. Let $A$ be a unital C*-algebra.
Denote by $CP(A)$ the set of cp maps from $A$ to $A$. To formulate selflessness of $\phi\in CP(A)$, we rely on Shlyakhtenko's construction of the C*-algebra  $C^*(A,s_{\phi})$ generated by $A$
and a semicircular element $s_{\phi}$ with covariance $\phi$. We call $\phi\in CP(A)$  selfless if the canonical embedding
\[
(A;\phi)\to (C^*(A,s_{\phi});\phi\circ E_\Omega)
\]
is positively existential. This means that for some ultrapower $A^{\cU}$ there is a *-homomorphism $\sigma\colon C^*(A,s_{\phi})\to A^{\cU}$ such that $\sigma|_A$ induces the diagonal embedding and $\sigma$ intertwines the cp maps $\phi\circ E_\Omega$ and $\phi^{\cU}$.

The case when the cp map is a conditional expectation is of special interest, and further illuminates selflessness in this new setting. Let $E\colon A\to B$ be a conditional expectation 
(onto a unital C*-subalgebra $B$) inducing a faithful GNS representation of $A$. Then $E$ is selfless if and only if there exists
 a C*-probability space $(C,\kappa)$, with $C\neq \C$ and $\kappa$ inducing a faithful GNS representation of $C$,  for which the embedding
$$
(A;E)\to (A;E)*_B (B\otimes C; \mathrm{id}_B\otimes \kappa)
$$
is existential (where the right-hand side is an amalgamated reduced free product, and the tensor product is the minimal one). Thus, there exists a *-homomorphism $\sigma\colon A*_B (B\otimes C)\to A^\cU$  whose restriction to $A$ agrees with the diagonal inclusion and that intertwines the expectations $E*_B(\mathrm{id}_B\otimes \kappa)$ and $E^\cU$. Note that, besides the free independence relation that this enforces
in $A^\cU$---$\sigma(B\otimes C)$ is free from $A$ with amalgamation over $B$---selflessness of $E$ enforces a commutation relation:
$\sigma(C)\subseteq B'\cap A^\cU$.  By choosing  $E$ to be a state (i.e., $B=\C$), we recover the notion of selfless C*-probability space.

To investigate selfless cp maps and expectations, we go one step further in generality and define selflessness for a  C*-correspondence $H$ (over $A$) endowed with a real structure $K\subseteq H$. To formulate selflessness for the pair $(H,K)$, we again appeal to Shlyakhtenko's  C*-algebra generated by $A$ and the $A$-valued semicircular system associated to the pair $(H,K)$ \cite{Shlyakhtenko1999Avalued}. Selfless cp maps are then precisely  those for which the C*-correspondence $A\otimes_{\phi} A$, obtained via the KSGNS construction, endowed with the real structure generated by $1\otimes 1$, is selfless. In this introduction, we will confine our discussion to the settings of cp maps and conditional expectations, as this allows for a less technical presentation.

We also work with a notion of  weak selflessness, which,
 as its name indicates, is implied by selflessness. A cp map $\phi\in CP(A)$ is called weakly
selfless if the embedding
$$
A\to C^*(A,s_{\phi})*_A C^*(A,s_{\phi})*_A \cdots
$$
is a positively existential embedding of C*-algebras,
where the right-hand side is the infinite amalgamated reduced free product of $(C^*(A,s_\phi), E_\Omega)$ with itself. 
Note that we do not enforce an intertwining of the  associated cp maps. 
Weak selflessness behaves more flexibly than selflessness  (e.g., the identity map $\mathrm{id}_A\colon A\to A$
can never be selfless but may be weakly selfless), but is strong enough  to have structural implications for the underlying C*-algebra. 

In the state case, if $(A,\rho)$ is selfless then it has the Dixmier property. This in turn implies that $A$ is simple, and that it is either traceless or has at most one trace (if $\rho$ is tracial). In the setting of cp maps$ $, we have the following:
\begin{proposition}
	Let $A$ be a unital C*-algebra.
	\begin{enumerate}[(i)]
\item 
If $\phi\in CP(A)$ is  weakly selfless, then it is  approximately selfadjoint inner, that is to say, it is a point-norm limit of cp maps of the form
$$
x\mapsto \sum_{i=1}^n h_ixh_i,
$$
where $h_i\in A$ is selfadjoint.
\item 
If $E\colon A\to B$ is a weakly selfless conditional inducing a faithful representation of $A$, then $\Delta_A\circ E\colon A\to A^{\cU}$ 
is a point norm limit of maps of the form
$$
x\mapsto \frac1n \sum_{i=1}^n u_i^*xu_i,
$$
where $u_i\in B'\cap A^{\cU}$ is a unitary for all $i$. (Here $\cU$ is an ultrafilter on $\N$ and $\Delta_A$  the diagonal inclusion of $A$ in its ultrapower $A^{\cU}$.)
\end{enumerate}
\end{proposition}
As a consequence, if $\phi$ is weakly selfless then it preserves closed two-sided ideals of $A$, and if it is also unital, then if preserves the tracial states of $A$.

In the state case,  a strong dichotomy holds for selfless C*-probability spaces: their underlying C*-algebras  are either purely infinite or have stable rank one \cite{robertselfless}, \cite{gould2026selfless}. Moreover, in the tracial state case, the C*-algebra has  strict comparison of positive elements by its unique trace. 

Although the purely infinite/stable-rank-one dichotomy does not persist for general cp maps, weak selflessness continues to impose regularity on the underlying algebra. We show that if a unital $C^*$-algebra $A$ admits a weakly selfless unital cp map, then  every bounded trace on $A^{\mathcal U}$ is essentially a limit trace (i.e., $A^{\cU}$ has no ``phantom traces''). The interactions between regularity properties and selflessness of expectations, cp maps, and C*-correspondences, will be the subject of future research.

To isolate the purely infinite case, we introduce  the notion of Toeplitz selflessness. We call $(A,\phi)$ Toeplitz selfless if the embedding
\[
(A;\phi)\to (C^*(A,\ell_{\phi});\phi\circ E_\Omega)
\] 
is positively existential, where this time $C^*(A,\ell_{\phi})$ is the Toeplitz-Pimsner C*-algebra associated to the KSGNS C*-correspondence $A\otimes_{\phi} A$, and
$\ell_\phi$ is the left creation operator. In the case  of an expectation $E$ inducing a  faithful GNS representation of $A$, Toeplitz selflessness is equivalent to simply asking that  the embedding
$$
(A;E)\to (A;E)*_B (B\otimes C; \mathrm{id}_B\otimes \kappa)
$$
is existential for some $C$ and a \emph{non-tracial} GNS-faithful state $\kappa$. Toeplitz selfless cp maps are \emph{completely} selfless.

\begin{theorem}
If a unital C*-algebra $A$ admits a  weakly Toeplitz selfless cp map $\phi\in CP(A)$ such that $\mathrm{Ideal}(\phi(a))=\mathrm{Ideal}(a)$ for all $a\in A_+$, then $A$ is strongly purely infinite (in the sense of Kirchberg and R{\o}rdam \cite{KirRorNonsimpleOinfty}). 
\end{theorem}	

A substantial part of the paper is devoted to permanence results for the new notions of selflessness. We mention some here:

\begin{theorem}\label{thm:intro-permanence}
Let $A$ be a unital C*-algebra.
\begin{enumerate}[(i)]
\item
If $\phi_n\in CP(A)$ for $n=1,\ldots$ are weakly selfless maps, and $\phi_n\to \phi$ in point norm, then $\phi$ is weakly selfless.
\item
If $\phi,\psi\in CP(A)$ commute, $\phi$ is selfless, and $\psi$ is approximately selfadjoint inner, then $\psi\circ\phi$ is selfless.
\item
If $\phi\in CP(A)$ is selfless, $\psi\in CP(B)$ is  approximately selfadjoint inner, and $B$ is exact, then $\phi\otimes \psi$ is selfless.
If $\phi$ is completely selfless, then same conclusion holds without the assumption that $B$ is exact.
\item
If $E_1\colon A_1\to B$ is a selfless conditional expectation, and $E_2\colon A_2\to B$ is a conditional expectation inducing a faithful GNS representation, then $E_1*_B E_2\colon A_1*_B A_2\to B$ is selfless. 
\end{enumerate}
\end{theorem}

In \cite{ozawa2025proximalityselflessnessgroupcalgebras}, Ozawa shows that a tensor product of exact selfless C*-algebras is selfless. Theorem \ref{thm:intro-permanence} (iii) extends this result to cp maps, while also relaxing the assumption on one of the cp maps to be approximately selfadjoint inner. Specializing to states, this shows that if $(A,\rho)$
is selfless, and $B$ is exact, unital, simple, with unique trace $\tau$,  then $(A\otimes B, \rho\otimes \tau)$ is selfless. This yields several new examples of $C^*$-selfless groups, including $G\times H$ where $G$ has Ozawa's PHP property (such as $C^*$-simple acylindrically hyperbolic or linear groups) and $H$ is $C^*$-simple.

 %In \cite{ozawa2025proximalityselflessnessgroupcalgebras}, Ozawa introduced the PHP property for discrete groups and showed that if $G$ has the PHP property then $C_r^*(G)$ is completely selfless. Classes of groups with the  PHP property include linear C*-simple groups and acylindrically hyperbolic C*-simple groups \cite{ozawa2025proximalityselflessnessgroupcalgebras, vigdorovich2026selflessreducedcalgebraslinear, yang2025extreme}. 
 %On the other hand, there exist C*-selfless groups (i.e., such that $C_r^*(G)$ is selfless) that fail to have the PHP property. Indeed, by the results discussed in the preceding paragraph,    if $G$ has the PHP property, and $H$ is C*-simple, then $G\times H$ is C*-selfless.  But such  direct products may fail to satisfy the PHP property (see Remark \ref{rem: no-PHP-selfless}). 

A natural setting to investigate selflessness of expectations is the reduced crossed products. By an adaptation of Ozawa's PHP technique we prove the following result:

\begin{theorem}
    If a group $G$ with the PHP property acts on a unital C*-algebra $A$ by approximately inner automorphisms, then the conditional expectation 
$E\colon A\rtimes_r G\to A$ is selfless.
\end{theorem}
Combined with the previous permanence results,
this yields new examples of selfless C*-algebras. For example, since every automorphism of the Jiang-Su algebra $\mathcal Z$
is approximately inner, the  expectation $E\colon \mathcal Z\rtimes_r G\to \mathcal Z$ is selfless for any action of a PHP group $G$ on $\mathcal Z$.
Further, since $\mathcal Z$ is simple with unique trace $\tau$, it follows by Theorem \ref{thm:intro-permanence} (ii) that the trace $\tau\circ E\colon \mathcal Z\rtimes_r G\to \C$  is selfless, i.e., $(\mathcal Z\rtimes_r G, \tau\circ E)$ is a selfless C*-probability space. This notably yields new examples of $C^*$-algebras with strict comparison.
  
Our final group of results concerns the interaction between selflessness and nuclearity. Here selflessness becomes a criterion for tensorial absorption.
Indeed, the purely infinite/unique trace-strict comparison dichotomy, together with the Kirchberg and Matui–Sato absorption theorems, implies that every nuclear selfless $C^*$-probability space has $\mathcal Z$-stable C*-algebra.

A direct link between selflessness of cp maps and  tensorial stability is that if  $A$ is separable and the identity map $\mathrm{id}_A$ is weakly selfless, then $A$
is $\mathcal Z$-stable, while if $A$ is separable and $\mathrm{id}_A$ is weakly Toeplitz selfless, then $A$ is $\mathcal O_\infty$-stable.  We deduce from this and the permanence results on selflessness  the following:

\begin{theorem}
Let $A$ be a separable unital C*-algebra and $\phi\in CP(A)$.
\begin{enumerate}[(i)]
\item 
If $\phi$ is both weakly selfless and real relatively nuclear (Definition \ref{def:relative-nuclearity}), then $A$ is $\mathcal Z$-stable.
\item 
If $\phi$ is both weakly Toeplitz selfless and relatively nuclear (Definition \ref{def:relative-nuclearity}), then $A$ is $\mathcal O_\infty$-stable.
\end{enumerate}	
\end{theorem}	
The proof of Ozawa's \cite[Theorem 3]{ozawa2025proximalityselflessnessgroupcalgebras} shows that if $A$ is simple, unital, and purely infinite, then $(A,\rho)$
is  Toeplitz selfless (for any state $\rho$). Combined  with (ii) of the previous theorem, this gives  a new proof of Kirchberg's $\mathcal O_\infty$-stability theorem (\cite{RordamStormer-EMS126}): if $A$ is separable, purely infinite, simple, unital, and nuclear, then $A\otimes \mathcal O_\infty\cong A$.

As another application of our results on selflessness and nuclearity, we establish a special case of a conjecture of Hayes \cite[Conjecture 1.2]{Schafhauser2026}. The conjecture asserts that if $G_1$ and $G_2$ are MF groups with a common amenable subgroup $H$, then the amalgamated free product $G_1 *_H G_2$ is MF. We prove that this is the case if $G_1=G_2$: if $G$ is an MF group and $H\leq G$ is amenable, then $G *_H G$ is MF.

\subsection*{Organization of the paper} The paper is organized as follows: In Sections \ref{sec: correspondences} and \ref{sec: weak containment} we develop the general framework of selflessness for C*-correspondences.
In Section \ref{sec: cp maps} we consider selfless  cp maps. 
In Section \ref{sec: regularity} we obtain some regularity properties for the underlying C*-algebra of a weakly selfless cp map.
In Section \ref{sec: expectations} we
specialize our results to conditional expectations. In Section \ref{sec: nuclearity} we investigate the interaction between selflessness and nuclearity. In Section \ref{sec: crossed products} we consider expectations arising from crossed products.

\subsection*{Acknowledgments} We are grateful to the Brin Mathematical Research Center for hosting the workshop ``Recent developments in operator algebras'' organized by the third author, wherein this work began.

\subsection*{AI statement} ChatGPT was used for English language editing, proofreading, and grammatical
corrections. All mathematical content was generated solely by the human authors.

\section{Preliminaries}
Throughout, all C*-algebras are assumed unital and all *-homomorphisms are unital. 

\subsection{C*-correspondences}
By an $A$-$B$ C*-correspondence we mean a right Hilbert $B$-module 
$H$ together with a unital *-homomorphism $A\to B(H)$, which we write as a left action.  (We denote by   $B(H)$ the C*-algebra of adjointable operators on $H$.)

We review some facts about C*-correspondences here. For more background, see \cite[Chapter 4]{Lance1995} and \cite{brown2008textrm}.

We call an $A$-$A$ C*-correspondence $H$ a C*-correspondence  over $A$. Note that $A$ is a C*-correspondence over itself with the obvious left and right actions and standard Hilbert C*-module inner product $\langle a,b\rangle =  a^*b$.

A covariant map $(\Phi,\sigma)\colon H_1\to H_2$ between C*-correspondences respectively over $A$ and $B$ consists of  a *-homomorphism $\sigma\colon A\to B$ and a linear map $\Phi\colon H_1\to H_2$ such that 
\begin{enumerate}
\item 
$\Phi(avb)=\sigma(a)\Phi(v)\sigma(b)$ for all $a,b\in A$ and $v\in H_1$,
\item 
$\sigma(\langle v,w\rangle  )=\langle \Phi(v),\Phi(w)\rangle$ for all $v,w\in H_1$.	
\end{enumerate}	
We often emphasize $\Phi$ and refer to $\sigma$ as the underlying *-homomorphism.
A covariant map $H\to B$ into the identity C*-correspondence of a C*-algebra $B$ is called a representation of $H$.

We often use the following lemma:
\begin{lemma}\label{lem:singly-generated-embedding}
Let $\sigma\colon A\to C$ be an embedding of C*-algebras. Let $H$ and $H'$ be C*-correspondences over $A$ and $C$, respectively. Suppose that $H$ is  
generated (as a closed bimodule) by an indexed collection of vectors $(v_i)_{i\in I}$, and that we have an assignment $v_i\mapsto w_i\in H'$ such  that
$$
\sigma(\langle v_i,av_j\rangle)=\langle w_i,\sigma(a)w_j\rangle
$$
for all $a\in A$ and all $i,j\in I$. Then there is a unique covariant embedding $H\to H'$ such that $v_i\mapsto w_i$ and with underlying *-homomorphism $\sigma$.
\end{lemma}

Given $A$-$B$ and $B$-$C$ C*-correspondences $H_1$ and $H_2$, $H_1 \otimes_B H_2$ is obtained by ``separation and completion'' from their algebraic tensor product with semi-inner product defined as
$$
\langle v_1\otimes w_1,v_2\otimes w_2\rangle=\langle w_1,\langle v_1,v_2\rangle w_2\rangle. 
$$

\subsection{The Toeplitz-Pimsner C*-algebra} 
We recall the notation for the Toeplitz-Pimsner algebra associated to a C*-correspondence. For background, see \cite[Proposition 3.3]{Katsura2004} and \cite[Section 4.6]{brown2008textrm}.

Let $H$ be a C*-correspondence over $A$. Its Fock C*-correspondence is
$$
\mathcal F(H)=A\oplus \bigoplus_{n\geq 1} H^{\otimes_A n}.
$$
For $v\in H$, let $\ell_v\in B(\mathcal F(H))$ denote the left
creation operator. Together with the left action of $A$ on $\mathcal F(H)$,
the assignment $v\mapsto \ell_v$ gives a representation of $H$ on
$\mathcal F(H)$. We denote by
$$
\cT(H)=C^*(A,\{\ell_v:v\in H\})
$$
the Toeplitz-Pimsner C*-algebra of $H$. We denote by 
$E_\Omega\colon \cT(H)\to A$ the vacuum expectation, obtained by
compressing to the $A$-summand:
\[
E_\Omega(x)=P_0xP_0.
\]
The representation of $H$ in $\cT(H)$ has the following universal property:
if $(\Phi,\sigma)$ is a representation of $H$ in a C*-algebra $D$, then there
is a *-homomorphism $\pi\colon \cT(H)\to D$ such that
\[
\pi(\ell_v)=\Phi(v),\qquad \pi(a)=\sigma(a)
\]
for all $v\in H$ and $a\in A$. 
See \cite[Proposition 6.5]{Katsura2004},  \cite[Theorem 4.6.18]{brown2008textrm}.

We will use the following application of the functoriality of  Toeplitz-Pimsner C*-algebra construction, which we informally refer to as ``Toeplitz exactness''; see \cite{gao-KE}. 

\begin{lemma}[Toeplitz exactness]\label{lem:toeplitz-exactness}
Let $\cU$ be an ultrafilter on $I$. Let $H$ be a C*-correspondence over $A$, 
and let $H_i$ be a C*-correspondence
over $A_i$, for each $i\in I$. Suppose that there is a covariant map
$$
(\Phi,\sigma)\colon H\longrightarrow \prod_{\cU}H_i,
$$
where $\sigma\colon A\to\prod_{\cU}A_i$ is the underlying *-homomorphism. Then there exists a *-homomorphism
$$
\widetilde{\sigma}\colon\cT(H)\longrightarrow
\prod_{\cU}\cT(H_i)
$$
extending $\sigma$ and satisfying $\widetilde{\sigma}(\ell_v)=[(\ell_{v_i})_i]$ whenever $\Phi(v)=[(v_i)_i]$. Moreover,
$$
E_{\cU}\circ\widetilde{\sigma} = \sigma\circ E_H,
$$
where $E_{\cU}\colon\prod_{\cU}\cT(H_i)\to \prod_{\cU}A_i$ is given by $E_{\cU}([(x_i)_i])=[(E_{H_i}(x_i))_i]$.

Suppose, in addition, that $H$ and the $H_i$ are equipped with real
structures $K$ and $K_i$, respectively, and that $\Phi(K)\subseteq\prod_{\cU}K_i$. 
Then $\widetilde{\sigma}$ restricts to a *-homomorphism
$$
\cS(H,K)\longrightarrow\prod_{\cU}\cS(H_i,K_i)
$$
intertwining the corresponding expectations. (For definition of real structures and $\cS$, see Section \ref{subsec shlyakhtenko}.)
\end{lemma}

\begin{proof}
See also \cite{gao-KE}. For each $i\in I$, let
$L_i\colon H_i\to\cT(H_i)$ be the canonical representation,
$v\mapsto\ell_v$. The map
$$
\prod_{\cU}H_i\longrightarrow\prod_{\cU}\cT(H_i),
\qquad
[(v_i)_i]\longmapsto[(\ell_{v_i})_i],
$$
together with the canonical *-homomorphism
$$
\prod_{\cU}A_i\longrightarrow\prod_{\cU}\cT(H_i),
$$
is a representation of $\prod_{\cU}H_i$. Precomposing this
representation with $(\Phi,\sigma)$ gives a representation of $H$
in $\prod_{\cU}\cT(H_i)$. The universal property of $\cT(H)$ now
gives $\widetilde{\sigma}$.

The identity
$$
E_{\cU}\circ\widetilde{\sigma}=\sigma\circ E_H
$$
is verified on the dense linear span of the elements
$\ell_\xi\ell_\eta^*$, where $\xi$ and $\eta$ range over tensor
powers of $H$. Finally, if $v\in K$ and
$\Phi(v)=[(v_i)_i]$ with $v_i\in K_i$, then
$$
\widetilde{\sigma}(s_v)=[(s_{v_i})_i].
$$
This proves the final assertion.
\end{proof}

\subsection{Shlyakhtenko's C*-algebra}\label{subsec shlyakhtenko}
Let $H$ be a correspondence over $A$.
A norm closed real subspace $K\subseteq H$ is called a \emph{real structure}  if 
\begin{enumerate}
\item 
$K$ generates $H$ as closed $A$-$A$ bimodule,
\item 
$K$ is closed under $v\mapsto a^*va$ for all $a\in A$.
\end{enumerate}
For the identity correspondence, one such real structure is the subspace of selfadjoint elements $A_{sa}$. If we assume that $A$ is unital (as we shall always do), this is the real structure generated by the unit. We call the pair $(H,K)$ a C*-correspondence with real structure. We sometimes write $(H,K,A)$ to indicate the underlying C*-algebra and remove ambiguity.

For covariant maps between C*-correspondences with real structures (possibly over different C*-algebras)
$(H_1,K_1)\to (H_2,K_2)$  we require that $K_1$ is mapped to $K_2$.

Let us now review Shlyakhtenko's C*-algebra of semicirculars associated to a C*-correspondence with real structure $(H,K)$
\cite{Shlyakhtenko1999Avalued}. 
The elements $s_v=\ell_v+\ell_v^*\in \cT(H)$	are called semicircular elements. 
Given $(H,K)$, we form the C*-algebra 
\[
\cS(H,K):=C^*(A,\{s_v:v\in K\})\subseteq \mathcal T(H),
\]
which we call the C*-algebra of $A$-valued semicirculars associated to $(H,K)$. We note that it comes equipped with the  vacuum expectation onto $A$. 

\begin{lemma}\label{lem:semicircular-jordan-generators}
If $\Sigma\subseteq K$ generates $K$ as a closed Jordan bimodule, then 
\[
\cS(H,K)=C^*(A,\{s_v:v\in \Sigma\}).
\]
\end{lemma}
\begin{proof}
This follows from the identities  $s_{rv+w}=rs_v+s_w$, $s_{ava^*}=as_va^*$, for $r\in \R$, $v,w\in K$, $a\in A$, and from the continuity of $v\mapsto s_v$. 	
\end{proof}	

\subsection{Indexed families of cb maps}
We sometimes work with pairs $(A;(\phi_i)_{i\in I})$ where $A$ is a unital
C*-algebra and $(\phi_i)_{i\in I}$ is an indexed family of completely bounded (cb) maps $\phi_i\in \mathrm{CB}(A)$. By a morphism $\sigma\colon (A;(\phi_i)_{i\in I})\to (B;(\psi_i)_{i\in I})$ between such objects  we understand a *-homomorphism $\sigma\colon A\to B$ that intertwines the cb maps: $\psi_i\sigma=\sigma\phi_i$ for all $i$ (notice that the index set is the same on both domain and codomain). 

By the ultrapower of $(A;(\phi_i)_{i\in I})$ we understand $(A^{\cU}; (\phi_{i}^{\cU})_{i\in I})$. We denote by $\Delta_A$ the diagonal embedding
$(A;(\phi_i)_i)\to (A^{\cU}; (\phi_{i}^{\cU})_{i\in I})$.

\begin{definition}\label{def:positively-existential}
Consider an embedding of C*-algebras with indexed families of cb maps:
$$
\theta\colon (A;(\phi_i)_{i\in I})\to (B;(\psi_i)_{i\in I}).
$$
We say that $\theta$ is
positively  existential if there exists an ultrafilter $\cU$ and a *-homomorphism
$$
\sigma\colon (B;(\psi_i)_{i\in I})\to (A^{\cU}; (\phi_i^{\cU})_{i\in I})
$$
such that $\sigma\theta=\Delta_A$. 

If $\sigma$ may be chosen to also  be an embedding, then we say that $\theta$ is  existential.

Following \cite[Section 6]{ozawa2025proximalityselflessnessgroupcalgebras}, we say that $\theta$ is completely positively existential 
if $\sigma$ may be chosen such that for every C*-algebra $C$ the induced *-homomorphism 
$\tilde\sigma\colon B\otimes C\to (A\otimes C)^\cU$ is well defined (where the tensor products are minimal). 
If $\sigma$ is additionally an embedding, we say that $\theta$ is completely existential.
\end{definition}

\section{Selfless C*-correspondences}\label{sec: correspondences}
Let $(H,K)$ be a $C^*$-correspondence over $A$ with a real structure. 
Denote by $\cS(H,K)$  Shlyakhtenko's  C*-algebra and by $E\colon \cS(H,K)\to A$  the vacuum conditional expectation.

For $v,w\in H$, set
$$
\kappa_{v,w}(a)=\langle v,aw\rangle,\qquad a\in A.
$$

\begin{definition}\label{def:selfless-correspondence} 
We call $(H,K)$ selfless if 
\begin{equation}\label{ssalpha}
    (A;\{\kappa_{v,w}:v,w\in H\})\to (\cS(H,K);\{\kappa_{v,w}\circ E:v,w\in H\})
\end{equation}
is a positively existential embedding of C*-algebras with indexed collections of cb maps. Equivalently, 
there exists  $\sigma\colon \cS(H,K)\to A^\cU$  whose restriction to $A$ agrees with the diagonal embedding $\Delta_A$ and such that
\begin{equation}\label{sssigma}
\kappa_{v,w}^\cU\circ \sigma = \Delta_A\circ \kappa_{v,w}\circ E \qquad \text{for all }v,w\in H.
\end{equation}	
If the embedding \eqref{ssalpha} is completely positively existential, then we call $(H,K)$ completely selfless.
\end{definition}

\begin{remark}
We prove all results for selflessness; the completely selfless versions are obtained by the same arguments.
\end{remark}

In the next three lemmas we clarify some points concerning this definition. We fix a C*-correspondence with real structure $(H,K)$,
set $C=\cS(H,K)$, and denote by $E\colon C\to A$ the vacuum expectation.

\begin{lemma}\label{lem:selfless-automatic-embedding}
If the left action of $A$ on $H$ is faithful,  and  a *-homomorphism  $\sigma\colon C\to A^{\cU}$ satisfies \eqref{sssigma}, then $\sigma$ is an embedding.
Thus, in this case the embedding \eqref{ssalpha} is existential.
\end{lemma}

\begin{proof}
Suppose that   $\sigma(c)=0$ for some  $c\in C$. Then 
$$
\Delta_A(\kappa_{v,w}(E(xcy)))=\kappa_{v,w}^\cU(\sigma(xcy))=0
$$
for all $x,y\in C$ and $v,w\in H$. By faithfulness of the left action of $A$ on $H$, this implies that $E(xcy)=0$ for all $x,y\in C$. Since $E$ has faithful GNS,  we get $c=0$. 
\end{proof}		

\begin{lemma}\label{lem:selfless-check-generators}
Let $\sigma\colon C\to A^{\cU}$ be a	*-homomorphism such that $\sigma|_A=\Delta_A$.	If $\Sigma\subseteq H$ generates $H$ as a closed bimodule, then $\sigma$ satisfies \eqref{sssigma} if and only if
$$
\kappa_{v,w}^\cU\circ \sigma = \Delta_A \circ \kappa_{v,w}\circ E\hbox{ for all }v,w\in \Sigma.
$$
Thus, to conclude  selflessness of $(H,K)$,  it is sufficient  to check that the embedding
$$
(A;\{\kappa_{v,w}:v,w\in \Sigma\})\to (C;\{\kappa_{v,w}\circ E:v,w\in \Sigma\})
$$
is positively existential.
\end{lemma}

\begin{proof}
The forward implication is immediate. For the converse, consider the set of all cb maps $\kappa\in CB(A)$ such that
$$
\kappa^\cU\circ \sigma=\Delta_A\circ \kappa\circ E.
$$
It is straightforward to check that this set is closed under linear combinations, operator-norm limits, pre-composition by the left and right multiplication maps
$$
L_a(x)=ax,\qquad R_a(x)=xa,
$$
and post-compositions by arbitrary cb maps on $A$. We now use that
\begin{align*}
\kappa_{v_1+v_2,w} &=\kappa_{v_1,w}+\kappa_{v_2,w}, \quad 
\kappa_{b_1vb_2,w}=L_{b_2^*}\circ \kappa_{v,w}\circ L_{b_1^*},\\
\kappa_{v,w_1+w_2}&=\kappa_{v,w_1}+\kappa_{v,w_2},\quad
\quad \kappa_{v,b_1wb_2}=R_{b_2}\circ \kappa_{v,w}\circ R_{b_1},
\end{align*}
to deduce that \eqref{sssigma} holds for $v,w$ ranging on a dense subset of $H$. Further, 
by  the norm continuity of $(v,w)\mapsto \kappa_{v,w}$, we obtain that \eqref{sssigma} holds for all $v,w\in H$.
\end{proof}

Consider the C*-correspondence over $C$ induced by $H$:
$$
H_C=L^2(C,E)\otimes_A H\otimes_A C.
$$
Endow $H_C$ with the real structure $K_C$ generated by the set $1\otimes K\otimes 1$. 
We note, for later use, that from the fact that $K$ generates $H$ as a closed bimodule,
we deduce that the vacuum vector is cyclic for the action of $C$ on $\cF(H)$. Hence, 
\begin{equation}\label{actsonFock}
L^2(C,E)\cong \cF(H)
\end{equation}
as $C$-$A$ C*-correspondences.

We have a covariant embedding
$(H,K,A)\to (H_C,K_C,C)$ where $\Phi\colon H\to H_C$ is defined as 
$$
\Phi(v)= 1\otimes v\otimes 1
$$
and $A$ is canonically embedded in $C$.

\begin{lemma}\label{lem:selfless-covariant-characterization} 
The pair $(H,K)$ is selfless if and only if the covariant embedding $(H,K,A)\to (H_C,K_C,C)$ is a positively  existential embedding of C*-correspondences with real structures. That is, for some ultrafilter $\cU$ there exists a covariant map 
$$
(\Psi,\sigma)\colon (H_C,K_C,C)\to (H^{\cU}, K^{\cU}, A^{\cU})
$$
such that 
$$
\Psi\Phi=\Delta_H, \quad \sigma|_{A}=\Delta_A.
$$
\end{lemma}	

\begin{proof}
Suppose that $(H,K)$ is selfless and let $\sigma\colon C\to A^\cU$ be a *-homomorphism witnessing this. Define $\Psi\colon H_C\to H^\cU$ on elementary tensors $x\otimes v\otimes y\in C\otimes_A H\otimes_A C$ by
$$
\Psi(x\otimes v\otimes y)=\sigma(x)\Delta_H(v)\sigma(y)
$$
and extend it by linearity to their linear span. The calculation
\begin{align*}
\langle\Psi(x\otimes v\otimes y),\Psi(x'\otimes w\otimes y')\rangle 
&=\sigma(y)^* \langle \Delta_H(v),\sigma(x^*x')\Delta_H(w)\rangle \sigma(y')\\
&=\sigma(y)^*\Delta_A(\langle v, E(x^*x')w \rangle)\sigma(y')\\
&=\sigma(y^*\langle v, E(x^*x')w\rangle y')\\
&=\sigma(\langle x\otimes v\otimes y, x'\otimes w\otimes y' \rangle)
\end{align*}
shows both that $\Psi$ is well defined and that it extends to a linear map from $H_C$ to $H^{\cU}$. It is immediate from its  definition that $(\Psi,\sigma)$ is covariant, that it sends $K_C$ into $K^{\cU}$, and that $\Psi\circ \Phi=\Delta_H$. Thus, the embedding $(H,K,A)\to (H_C,K_C,C)$ is positively existential.

Conversely, suppose that the embedding $(H,K,A)\to (H_C,K_C,C)$ is positively existential. Let
$\sigma\colon C\to A^{\cU}$ and $\Psi\colon H_C\to H^{\cU}$
witness this. Then $\sigma|_A=\Delta_A$ and $\Psi\circ\Phi=\Delta_H$. For $c\in C$ and $v,w\in H$, the covariance of 
$(\Psi,\sigma)$ gives
\begin{align*}
\kappa_{v,w}^{\cU}(\sigma(c))
&=\langle \Delta_H(v),\sigma(c)\Delta_H(w)\rangle\\
&=\langle \Psi(1\otimes v\otimes 1),\Psi(c\otimes w\otimes 1)\rangle\\
&=\Delta_A\big(\langle 1\otimes v\otimes 1,c\otimes w\otimes 1\rangle\big)\\
&=\Delta_A(\langle v,E(c)w\rangle)\\
&=\Delta_A(\kappa_{v,w}(E(c))).
\end{align*}
Thus $\sigma$ satisfies \eqref{sssigma}, and so $(H,K)$ is selfless.
\end{proof}

\begin{theorem}
\label{thm:selfless-subcorrespondences-directed-unions}
If $(H_1,K_1)\subseteq (H_2,K_2)$ and $(H_2,K_2)$ is (completely) selfless, then $(H_1,K_1)$ is (completely) selfless. If $\{(H_i,K_i)\}_{i\in I}$ is an upward directed family of (completely) selfless $A$-$A$ C*-subcorrespondences  of $(H,K)$ with dense union in $(H,K)$, then $(H,K)$ is also selfless.	
\end{theorem}

\begin{proof}
Suppose that $(H_1,K_1)\subseteq (H_2,K_2)$, with  $(H_2,K_2)$ selfless. The inclusion  $H_1\subseteq H_2$ induces a natural embedding $\cF(H_1)\hookrightarrow \cF(H_2)$. Identify $\cF(H_1)$ with a sub-C*-correspondence of $\cF(H_2)$ in this manner. Since $\cF(H_1)$ is invariant under left creation operators $\ell_v\in \cT(H_2)$ such that $v\in H_1$ and under their adjoints, the restriction $x\mapsto x|_{\cF(H_1)}$ yields the 
isomorphism 
$$
C^*(A, \{\ell_v\in \cT(H_2):v\in H_1\})\cong \cT(H_1).
$$
Identify $\cT(H_1)$ with a C*-subalgebra of $\cT(H_2)$ in this manner. Then,   since $K_1\subseteq K_2$,  we obtain 
$$
\cS(H_1,K_1) = C^*(A, \{s_v\in \cS(H_2,K_2):v\in K_1\}).
$$
Now if a *-homomorphism $\sigma\colon \cS(H_2,K_2)\to A^\cU$ fulfills \eqref{sssigma} for $(H_2,K_2)$, then its restriction to $\cS(H_1,K_1)$ fulfills \eqref{sssigma} for $(H_1,K_1)$.

Let $\{(H_i,K_i)\}_{i\in I}$ be an upward directed family with dense union in $(H,K)$. Identify $\cS(H_i,K_i)$ as C*-subalgebras of $\cS(H,K)$ as in the previous paragraph. Notice that they form an upward directed family of C*-subalgebras  with dense union. For each $i$ we have a *-homomorphism $\sigma_i\colon \cS(H_i,K_i)\to A^{\cU}$ fulfilling \eqref{sssigma} (where we pick a common ultrapower for all $\sigma_i$). Pick an ultrafilter $\mathcal V$ in $I$ containing the tail sets $\{i\geq i_0\}$. For $x\in \bigcup_i \cS(H_i,K_i)$, define
\[
\sigma(x)=[(\sigma_i(x))_{i}]\in (A^{\cU})^{\mathcal V},
\]
where $[(y_i)_i]$ on the right-hand side denotes the 
class of $(y_i)_i$ in the ultrapower. (Note that $\sigma_i(x)$
is defined for $\mathcal V$-almost all $i$.)
This is a well defined and contractive *-homomorphism on $\bigcup_i \cS(H_i,K_i)$, so it extends to $\sigma\colon \cS(H,K)\to (A^{\cU})^{\mathcal V}$. The verification that $\sigma$ satisfies \eqref{sssigma}, and thus witnesses selflessness of $(H,K)$ is left to the reader.
\end{proof}

In the next lemma we examine what happens when we iterate the construction of Shlyakhtenko's  C*-algebra of semicirculars. We first introduce the set-up for the lemma. 

Let $(H_1,K_1)$ and $(H_2,K_2)$ be $C^*$-correspondences over $A$ with real structures. Form 
$$
C_1=\cS_A(H_1,K_1)=C^*(A,\{s_v^{(1)}:v\in K_1\}),
$$
with vacuum expectation $E_1\colon C_1\to A$. Consider the $C_1$-$C_1$ C*-correspondence induced by $H_2$:
$$
\widetilde H_2=L^2(C_1,E_1)\otimes_A H_2\otimes_A C_1,
$$
endowed with the real structure $\widetilde K_2$ generated by $1\otimes K_2\otimes 1$. 
Now form
\begin{align*}
\cS_{C_1}(\widetilde H_2,\widetilde K_2) &=C^*(C_1,\{s_{1\otimes w\otimes 1}^{(2)}:w\in K_2\})\\
&=C^*(A,\{s_v^{(1)},s_{1\otimes w\otimes 1}^{(2)}:v\in K_1,\, w\in K_2\}),
\end{align*}
with vacuum expectation $\widetilde E_2\colon \cS_{C_1}(\widetilde H_2,\widetilde K_2)\to C_1$. Endow $\cS_{C_1}(\widetilde H_2,\widetilde K_2)$ with the expectation  $E_1\widetilde E_2$ onto $A$.

\begin{lemma}\label{lem:semicircular-iterated}
We have canonical isomorphisms
$$
\cS_{C_1}(\widetilde H_2,\widetilde K_2)
\cong \cS_A(H_1\oplus H_2,K_1\oplus K_2)\cong
\cS_A(H_1,K_1) *_A\cS_A(H_2,K_2)
$$
which intertwine the expectations onto $A$. Moreover, under the first isomorphism the semicircular
generators are mapped as follows:
\begin{align*}
s_v^{(1)} &\mapsto s_{v\oplus 0},\qquad v\in K_1,\\
s_{1\otimes w\otimes 1}^{(2)}&\mapsto s_{0\oplus w},\qquad w\in K_2.
\end{align*}
\end{lemma}

\begin{proof}
Set $D=\cS_{C_1}(\widetilde H_2,\widetilde K_2)$.
Let us show that 
$$
L^2(D,E_1\widetilde E_2)
\cong \cF_A(H_1\oplus H_2)
$$
as  C*-correspondences over $A$, via an explicit unitary.

Recall that, as remarked in \eqref{actsonFock}, 
$$
L^2(C_1,E_1)\cong \cF_A(H_1),\quad L^2(D,\widetilde E_2)\cong \cF_{C_1}(\widetilde H_2),
$$
where the first isomorphism is an isomorphism of $C_1$-$A$ C*-correspondences, and the second one  of $D$-$C_1$ C*-correspondences. From the first isomorphism we get that
$$
\widetilde H_2\cong\cF_A(H_1)\otimes_A H_2 \otimes_A C_1
$$
as $C_1$-$C_1$ C*-correspondences. From the second one we get that
\begin{align*}
L^2(D,E_1\widetilde E_2)
 &\cong
\cF_{C_1}(\widetilde H_2)\otimes_{C_1}L^2(C_1,E_1)\\
&\cong\cF_{C_1}(\widetilde H_2)\otimes_{C_1}\cF_A(H_1)
\end{align*}
as $D$-$A$ C*-correspondences.

Since $\widetilde H_2=\cF_A(H_1)\otimes_A H_2 \otimes_A C_1$, the $n$-th tensor power $\widetilde H_2^{\otimes_{C_1}n}$ has the form
$$
\cF_A(H_1)\otimes_A H_2 \otimes_A \cF_A(H_1)\otimes_A H_2
\otimes_A \cdots \otimes_A \cF_A(H_1)\otimes_A H_2\otimes_A C_1,
$$
where $\cF_A(H_1)$ and $H_2$ appear $n$ times.  Using now that $L^2(D,E_1\widetilde E_2)\cong \cF_{C_1}(\widetilde H_2)\otimes \cF_A(H_1)$, 
we get that 
$$
L^2(D,E_1\widetilde E_2) \cong \bigoplus_{n=0}^\infty \cF_A(H_1)\otimes_A H_2 \otimes_A \cF_A(H_1)\otimes_A H_2 \cdots \otimes_A \cF_A(H_1),
$$
where $\cF_A(H_1)$ appears $n+1$ times and $H_2$ appears $n$ times in the $n$-th summand on the right-hand side. By expanding each tensor factor $\cF_A(H_1)$ into a direct sum of tensor powers of $H_1$, we can identify
the infinite direct sum on the right-hand side with $\cF_A(H_1\oplus H_2)$. The unitary 
$$
U\colon L^2(D,E_1\widetilde E_2)\to \cF_A(H_1\oplus H_2)
$$
is defined on elementary tensors by interpreting each $H_1$-word as a block of $H_1$-letters and each $H_2$-factor as an $H_2$-letter, thereby obtaining a tensor word in $\cF_A(H_1\oplus H_2)$.

Under $U$, the copy of $C_1$ in $D$ acts as the canonical copy of $\cS_A(H_1,K_1)$ in $\cS_A(H_1\oplus H_2,K_1\oplus K_2)$. Thus, for $v\in K_1$,
$$
U s_v^{(1)}U^*=s_{v\oplus 0}.
$$
Also, for $w\in K_2$, the creation operator associated to $1\otimes w\otimes 1\in \widetilde H_2$ is carried by $U$ to the creation operator associated to $0\oplus w\in H_1\oplus H_2$, and likewise for the adjoints of these operators. Hence
$$
Us_{1\otimes w\otimes 1}^{(2)}U^*=s_{0\oplus w}.
$$
Therefore
$$
UDU^*=\cS_A(H_1\oplus H_2,K_1\oplus K_2),
$$
and $E_1\widetilde E_2$ corresponds to the vacuum expectation. This gives the first isomorphism.

For the second isomorphism,  we recall that  by a result of Speicher,  there is  an expectation preserving isomorphism  
$$
\cT_A(H_1\oplus H_2)\cong \cT_A(H_1)*_A\cT_A(H_2).
$$
\cite[Theorem 2.4]{BrownDykemaShlyakhtenko2002}.
On left creation operators, this isomorphism makes the assignments:
\begin{align*}
\ell_{v\oplus 0} &\to \ell_v\in \cT_A(H_1),\qquad v\in H_1,\\
\ell_{0\oplus w} &\to \ell_w\in \cT_A(H_2),\qquad w\in H_2.
\end{align*}
Restricting to the C*-subalgebras generated by semicircular elements,
we obtain the required isomorphism. This completes the proof.
See \cite[Proposition 2.18]{Shlyakhtenko1999Avalued}
for the von Neumann  algebras version.
\end{proof}

\begin{theorem}\label{thm:l2-selfless-correspondence}
If $(H,K)$ is selfless, then so is $(\ell^2(H), \ell^2(K))$.
\end{theorem}

\begin{proof}
By  Theorem \ref{thm:selfless-subcorrespondences-directed-unions}, it suffices to show that if $(H,K)$ is selfless, then so is $(H\oplus H,K\oplus K)$.
Let us prove this.

As in the paragraphs above, we set $C=\cS(H,K)$, let $E\colon C\to A$
be the vacuum expectation, and denote by $(H_C,K_C)$ the C*-correspondence over $C$ induced by $(H,K)$. Since $(H,K)$ is selfless, we have, by Lemma \ref{lem:selfless-covariant-characterization},
 a covariant map $\Psi\colon H_C\to H^{\cU}$ with underlying *-homomorphism 
 $\sigma\colon C\to A^{\cU}$ satisfying that $\Psi\Phi=\Delta_H$ and $\sigma|_A=\Delta_A$. By Toeplitz exactness (Lemma \ref{lem:toeplitz-exactness}),  there is a *-homomorphism $\tilde\sigma\colon \cT(H_C)\to \cT(H)^{\cU}$
 satisfying that
 \begin{itemize}
 \item 
 $\tilde\sigma|_{C}=\sigma$,
\item 
$\tilde\sigma$ intertwines expectations, i.e., 
$$
E^{\cU}\circ \tilde\sigma=\sigma\circ \tilde E,
$$
where $\tilde E\colon \cT(H_C)\to C$ is the vacuum expectation.

\item 
if $\tilde v=\Phi(v)=1\otimes v \otimes 1$, then
$\tilde\sigma(\ell_{\tilde v})$ is the image of the constant sequence $(\ell_v)_i$ in $\cT(H)^{\cU}$, and consequently
$\tilde\sigma(s_{\tilde v})=\Delta_C(s_v)$.
 \end{itemize}

If $v\in K$ and  $c\in C$, and we set
 $\tilde v=c^*\otimes v\otimes c$, then
 $$
 \tilde\sigma(s_{\tilde v})=\sigma(c)^*\Delta_C(s_v)\sigma(c)\in C^{\cU}.
 $$
 Thus, the restriction of $\tilde\sigma$
to  
$$
D=C^*(C,\{s_{\tilde v}:\tilde v\in K_C\})\cong \cS(H\oplus H,K\oplus K)
$$
ranges in $C^{\cU}$. 
Composing with $\sigma^\cU$, we get $\sigma^\cU\tilde \sigma|_D\colon D\to (A^\cU)^\cU$ that restricts to the diagonal embedding on $A$. 

Let us check that $\sigma^{\cU}\tilde\sigma|_D$ intertwines the coefficient
maps of $H\oplus H$ as in \eqref{sssigma}. Let $v,w\in H$. By selflessness of $(H,K)$,
$$
\kappa_{v,w}^{\cU}\circ\sigma = \Delta_A\circ \kappa_{v,w}\circ E.
$$
Applying this identity in the ultrapower and using
$$
E^{\cU}\circ \tilde\sigma=\sigma\circ \tilde E,
$$
we get
\begin{align*}
(\kappa_{v,w}^{\cU})^{\cU}\circ(\sigma^{\cU}\tilde\sigma)
&=
\Delta_A^{\cU}\circ \kappa_{v,w}^{\cU}\circ E^{\cU}\circ\tilde\sigma\\
&=
\Delta_A^{\cU}\circ \kappa_{v,w}^{\cU}\circ \sigma\circ \tilde E\\
&=
\Delta_A^{\cU}\circ \Delta_A\circ \kappa_{v,w}\circ E\circ \tilde E.
\end{align*}
Identifying the iterated diagonal embedding with the diagonal embedding into
$(A^\cU)^\cU$, this gives
\[
(\kappa_{v,w}^{\cU})^{\cU}\circ(\sigma^{\cU}\tilde\sigma)
=
\Delta_A\circ \kappa_{v,w}\circ E\circ \tilde E.
\]
Under the identification
\[
D\cong \cS(H\oplus H,K\oplus K),
\]
from Lemma \ref{lem:semicircular-iterated},
the expectation onto $A$ is $E\circ \tilde E$. For the
subspaces $H\oplus 0$ and $0\oplus H$, the coefficient maps  are $\kappa_{v,w}$, with $v,w\in H$, while the cross coefficient maps are zero. Hence Lemma \ref{lem:selfless-check-generators}
shows that $\sigma^{\cU}\tilde\sigma|_D$ satisfies \eqref{sssigma}, witnessing
selflessness of $(H\oplus H,K\oplus K)$.
\end{proof}

Next we formulate a simpler, but much more restrictive, variation on selflessness.

\begin{definition}\label{def:toeplitz-selfless-correspondence}
Let $H$ be a C*-correspondence over $A$.
We call $H$ Toeplitz  selfless if the embedding 	
\begin{equation}\label{sstoeplitz}
(A;\{\kappa_{v,w}:v,w\in H\})\to (\cT(H);\{\kappa_{v,w}\circ E:v,w\in H\})
\end{equation}
is positively existential.
\end{definition}	

Note that if $H$ is Toeplitz selfless, then $(H,K)$ is selfless with $K=H$, but the latter property is in principle weaker. Lemmas \ref{lem:selfless-automatic-embedding} and \ref{lem:selfless-check-generators} remain valid in this case.
 Similarly to Lemma \ref{lem:selfless-covariant-characterization}, $H$ is Toeplitz selfless if and only if the covariant embedding
$$
H\to L^2(\cT(H),E)\otimes_A H\otimes_A \cT(H), \qquad v\mapsto 1\otimes v\otimes 1,
$$
is a positively existential embedding of C*-correspondences. 
The version of Theorem \ref{thm:selfless-subcorrespondences-directed-unions} for Toeplitz selflessness  is already implicitly proved in that same theorem. We also have the following:

\begin{lemma}
If $H$ is Toeplitz selfless, then the embedding \eqref{sstoeplitz}
is completely positively existential (Definition \ref{def:positively-existential}). Thus, the notions of Toeplitz selfless and completely Toeplitz selfless agree.
\end{lemma}

\begin{proof}
Let $H$ be Toeplitz selfless, and  $\sigma\colon \cT(H)\to A^{\cU}$ be such that $\sigma|_A=\Delta_A$.

Let $C$ be a unital C*-algebra. Let $\theta\colon A\to A\otimes C$ denote the first factor embedding $a\mapsto a\otimes 1$ into the minimal tensor product $A\otimes C$.  
Let $H\boxtimes C$ be the exterior tensor product of $H$ with the identity C*-correspondence $C$, over  $A\otimes C$  (see \cite[Chapter 4]{Lance1995}). 

We obtain a representation of $H\boxtimes C$ in $(A\otimes C)^{\cU}$, with underlying *-homomorphism the diagonal embedding $\Delta_{A\otimes C}$, by mapping  $v\otimes 1$ to $\sigma(\ell_v)\in A^{\cU}$ followed by the embedding  $\theta^{\cU}\colon A^{\cU} \to (A\otimes C)^{\cU}$:
\begin{equation}\label{repHboxtimesC}
H\boxtimes C\ni v\otimes 1\mapsto \theta^{\cU}\sigma(\ell_v)\in (A\otimes C)^{\cU}.
\end{equation}

Indeed, using that
$$
\Delta_A(\langle v,aw\rangle) =\sigma(\ell_v)^*\Delta_A(a)\sigma(\ell_w)
$$
for all $a\in A$, and that the range of $\theta^\cU$ commutes with $\Delta_{A\otimes C}(1\otimes C)$ in $(A\otimes C)^\cU$, we get
\begin{align*}
\Delta_{A\otimes C}(\langle (v\otimes 1),(a\otimes c)(w\otimes 1)\rangle) &=\Delta_{A\otimes C}(\langle v,aw\rangle\otimes c)\\
&=(\theta^{\cU}\sigma(\ell_v))^*\Delta_{A\otimes C}(a\otimes c)(\theta^{\cU}\sigma(\ell_w)).
\end{align*}
Since the set $\{v\otimes 1:v\in H\}$ generates $H\boxtimes C$
as an $A\otimes C$-$A\otimes C$ correspondence, the assignment
\eqref{repHboxtimesC} extends to a  representation of $H\boxtimes C$
in $(A\otimes C)^{\cU}$. 

From the  representation of $H\boxtimes C$
in $(A\otimes C)^{\cU}$ we obtain, by the universal property of the Toeplitz-Pimsner C*-algebra, a *-homomorphism
$$
\tilde\sigma\colon \cT(H\boxtimes C)\to (A\otimes C)^{\cU}
$$
such that $\tilde \sigma(\ell_{v\otimes 1})=\theta^{\cU}\sigma(\ell_v)$
and $\tilde\sigma|_{A\otimes C}=\Delta_{A\otimes C}$. On the other hand  
$$
\cT(H\boxtimes C)\cong \cT(H)\otimes C
$$
canonically via the assignment $\ell_{v\otimes c}\mapsto \ell_{v}\otimes c$ (\cite[Lemma 4.6.24]{brown2008textrm}). We thus obtain  $\tilde\sigma\colon \cT(H)\otimes C\to (A\otimes C)^{\cU}$ such that 
$\tilde\sigma(x\otimes 1)=\theta^{\cU}\sigma(x)$
and $\tilde\sigma|_{A\otimes C}=\Delta_{A\otimes C}$. This shows that the embedding $A\to \cT(H)$ is completely positively existential, in the sense of Definition \ref{def:positively-existential} and of \cite[Section 6]{ozawa2025proximalityselflessnessgroupcalgebras}.
\end{proof}

The proof of the following theorem can be adapted from the proof of Theorem \ref{thm:l2-selfless-correspondence}, and it is simpler:

\begin{theorem}\label{thm:l2-toeplitz-selfless-correspondence}
If $H$ is Toeplitz selfless, then so is $\ell^2(H)$.	
\end{theorem}

\section{Weak selflessness and weak containment}\label{sec: weak containment}
\begin{definition}\label{def:weak-selfless-correspondence}
We call $(H,K)$ (completely) weakly selfless if the embedding
$$
A\to \cS(\ell^2(H),\ell^2(K))
$$
is a (completely) positively existential embedding of C*-algebras. 
In other words, if there exists a *-homomorphism $\sigma\colon \cS(\ell^2(H), \ell^2(K))\to A^{\cU}$ such that $\sigma|_A=\Delta_A$.

We call $H$  weakly Toeplitz selfless if the embedding
$$
A\to \cT(\ell^2(H))
$$
is a positively existential embedding of C*-algebras. 
\end{definition}

By Theorems \ref{thm:l2-selfless-correspondence} and \ref{thm:l2-toeplitz-selfless-correspondence}, (Toeplitz) selflessness implies weak (Toeplitz) selflessness.

\begin{definition}\label{def:weak-containment-correspondences}
Let $(H_1,K_1)$ and $(H_2,K_2)$ be C*-correspondences over $A$ with real structures. We say that 
$(H_1,K_1)$ is weakly contained in $(H_2,K_2)$, denoted by  $(H_1,K_1)\precsim (H_2,K_2)$, if 
for every pair of finite sets $F \subset A$, $G \subset K_1$, and $\epsilon > 0$, there exists $n \in \N$ and an assignment
\[
G\ni v \mapsto (\xi_1(v), \cdots, \xi_n(v))\in K_2^n
\]
such that
\begin{equation*}
	\Big\|\langle v, aw\rangle - \sum_{k = 1}^n \langle\xi_k(v), a\xi_k(w)\rangle\Big\| < \epsilon
\end{equation*}
for all $v,w \in G$ and $a \in F$. If $K_1=H_1$ and $K_2=H_2$ then we say that $H_1$
is weakly contained in $H_2$, and write $H_1\precsim H_2$.
\end{definition}

It is not hard to show that weak containment can be reformulated as follows:

\begin{lemma}\label{lem:weak-containment-ultrapower}
We have $(H_1,K_1)\precsim (H_2,K_2)$ if and only if for some ultrafilter $\cU$ there is a covariant embedding 
$$
(\Phi,\Delta_A)\colon (\ell^2(H_1),\ell^2(K_1),A)\to ((\ell^2(H_2))^{\cU},(\ell^2(K_2))^{\cU}, A^{\cU}).
$$
\end{lemma}

In the following theorem $\std$ denotes the  real structure 
on $A$, viewed as the identity C*-correspondence, generated by $1$. Thus,  $\std=A_{sa}$.

\begin{theorem}\label{thm:asainner-correspondence}
If $(H,K)$ is weakly selfless, then $(H,K)\precsim (A,\std)$.
\end{theorem}

\begin{proof}
Let $C=\cS(\ell^2(H),\ell^2(K))$, and denote by $E\colon C\to A$ the vacuum expectation. Fix finite sets
$$
\{v_1,\ldots,v_n\}\subseteq K, \qquad F\subseteq A,
$$
and let $\epsilon>0$. Set $\kappa_{ij}(a)=\langle v_i,av_j\rangle$.

For each $i$ and each $k\in\N$, let $v_i^{(k)}\in\ell^2(K)$ be the copy of $v_i$ in the $k$-th summand of $\ell^2(H)$, and let $s_i^{(k)}=s_{v_i^{(k)}}\in C$. After identifying $C\cong \cS(H,K)^{*_A\infty}$ (Lemma \ref{lem:semicircular-iterated}), the elements
$$
s_i^{(k)}as_j^{(k)}-\kappa_{ij}(a)
$$
are centered and belong to freely independent copies of
$\cS(H,K)$ in $C$. By Voiculescu's inequality \cite[Proposition 7.4]{junge2005picuineq},
$$
\Big\|
\frac{1}{N}\sum_{k=1}^N \Big(s_i^{(k)}as_j^{(k)}-\kappa_{ij}(a)\Big)
\Big\| = O\Big(N^{-1/2}\Big)\|a\|.
$$
Choose $N$ sufficiently large such that
$$
\Big\|
\frac{1}{N}\sum_{k=1}^N s_i^{(k)}as_j^{(k)} -\kappa_{ij}(a)
\Big\| \leq \frac{\epsilon}{2}\|a\| 
$$
for all $a\in F$ and all $1\leq i,j\leq n$.

Since $(H,K)$ is weakly selfless, there exists a *-homomorphism
$\sigma\colon C\to A^{\cU}$ whose restriction to $A$ is the diagonal embedding. We thus have 
$$
\Big\|
\frac{1}{N}\sum_{k=1}^N \sigma\Big(s_i^{(k)}\Big)a\sigma\Big(s_j^{(k)}\Big)  \kappa_{ij}(a)
\Big\| \leq \frac{\epsilon}{2}\|a\|
$$
for all $a\in F$ and all $1\leq i,j\leq n$. Choose selfadjoint lifts of the elements
$\sigma(s_i^{(k)})$ in the product algebra, and then pass to a single coordinate  to obtain selfadjoint
elements $h_i^{(k)}\in A$, for  $i=1,\ldots,n$ and $k=1,\ldots,N$,
such that
$$
\Big\|
\frac{1}{N}\sum_{k=1}^N h_i^{(k)}ah_j^{(k)} -\kappa_{ij}(a)
\Big\|
\leq  \epsilon\|a\|
$$
for all $a\in F$ and all $1\leq i,j\leq n$.
Thus, the assignment
$$
v_i\longmapsto
\Big(
\frac{1}{\sqrt{N}}h_i^{(1)},\ldots,\frac{1}{\sqrt{N}}h_i^{(N)}
\Big)
$$
yields $(H,K)\precsim(A,\std)$.
\end{proof}

Given a C*-correspondence $H$ over $A$, let us call $\alpha\in CB(A)$ an  extended coefficient of $H$ if it can be obtained from the coefficient maps 
$\{\kappa_{v,w}:v,w\in H\}$ by linear combinations, operator norm limits, pre-compositions by 
left/right multiplication operators $x\stackrel{M_{a,b}}{\mapsto} axb$, and post-compositions by cb maps.
That is, $\alpha$ is an operator norm limit of maps of the form 
\[
\sum_i \beta_i \circ \kappa_{v_i,w_i}\circ M_{a_i,b_i}
\]
for $\beta_i\in CB(A)$, $v_i,w_i\in H$, $a_i,b_i\in A$.

\begin{theorem}\label{thm:weak-containment-permanence-selfless}
Suppose that $(H_1,K_1)\precsim (H_2,K_2)$.
\begin{enumerate}[(i)]
\item If $(H_2,K_2)$ is weakly selfless, then so is $(H_1,K_1)$.
\item If $(H_2,K_2)$ is selfless, and  moreover for each $v,w\in H_1$,   $\kappa_{v,w}$ belongs to the set of  extended coefficients of $H_2$,  then $(H_1,K_1)$ is selfless.	
\end{enumerate}
\end{theorem}

\begin{proof}
By Theorem \ref{thm:l2-selfless-correspondence} (or the definition of weakly selfless), we may replace
$(H_2,K_2)$ by $(\ell^2(H_2),\ell^2(K_2))$  and rename it as $(H_2,K_2)$.
    
(i) From the covariant embedding $(H_1,K_1,A)\to (H_2^\cU,K_2^\cU,A^\cU)$  we get by Toeplitz exactness
(Lemma \ref{lem:toeplitz-exactness}) a *-homomorphism  $\cT(H_1)\to \cT(H_2)^\cU$ intertwining the  expectations
$E_1$ and $E_2^{\cU}$, where $E_1$ and $E_2$ denote the vacuum expectations on $\cT(H_1)$ and $\cT(H_2)$, respectively. Restricting this *-homomorphism to the C*-algebra of semicircular elements, we get a *-homomorphism 
$$
\sigma_1\colon \cS(H_1,K_1)\to (\cS(H_2,K_2))^\cU.
$$
(The analogous statement for the  
enveloping von Neumann of semicircular elements  is \cite[Theorem C]{IoanaTan2024}.)

By weak selflessness of $(H_2,K_2)$, we have a *-homomorphism $\sigma_2\colon \cS(H_2,K_2)\to A^{\cU}$ agreeing with the diagonal inclusion on $A$. Then, the *-homomorphism   
$$
\sigma_2^\cU\sigma_1\colon \cS(H_1,K_1)\to (A^{\cU})^{\cU}
$$
agrees with the diagonal inclusion on $A$. Thus, $(H_1,K_1)$ is weakly selfless.

(ii) Let $v,w\in H_2$. Since $\sigma_1$ intertwines the expectations $E_1$ and $E_2^{\cU}$,  we have that
$$
\Delta_A \circ \kappa_{v,w}\circ E_1 =\kappa_{v,w}^\cU\circ E_{2}^{\cU}\sigma_1. 
$$
By selflessness of $(H_2,K_2)$, $\sigma_2$ intertwines $\kappa_{v,w}\circ E_{2}$ and $\kappa_{v,w}^{\cU}$.  We deduce that  $\sigma=\sigma_2^\cU\sigma_1$ intertwines $\kappa_{v,w}\circ E_{1}$ and  $(\kappa_{v,w}^{\cU})^{\cU}$. That is, 
\[
(\kappa_{v,w}^{\cU})^{\cU}\circ \sigma = \Delta_A\circ \kappa_{v,w}\circ E_1
\]
for all $v,w\in H_2$. Since the set of cb maps $\alpha\in CB(A)$ that satisfy
\[
(\alpha^{\cU})^{\cU}\sigma=\Delta_A\circ \alpha \circ E_1
\]
is closed under linear combinations, operator norm limits, pre-compositions by left/right multiplication operators, and post-compositions (see the proof of Lemma \ref{lem:selfless-check-generators}),
we deduce that the above equality holds for any extended coefficient of $H_2$.
In particular, it holds for the coefficient maps of $H_1$.  The desired result follows.
\end{proof}	

We state a more general version:
\begin{theorem}\label{thm:ultraproduct-permanence-selfless}
Suppose that we have a covariant embedding $(H,K)\to \prod_{\cU} (H_i,K_i)$ with underlying *-homomorphism $\Delta_A$.
\begin{enumerate}[(i)]
\item
If $(H_i,K_i)$ is weakly selfless for all $i$, then  $(H,K)$ is weakly selfless.
\item
If $(H_i,K_i)$ is selfless for all $i$ and the coefficient maps of $(H,K)$ belong to the extended coefficients set of $\cU$-almost all $(H_i,K_i)$. Then $(H,K)$
is selfless.
\end{enumerate}
\end{theorem}

\begin{remark} Assuming separability of $A$, we could state this theorem differently by asking that $\lim_{\cU} (H_i,K_i)=(H,K)$ in the Fell topology.
\end{remark}

\begin{proof}
From the covariant embedding $(H,K)\to \prod_{\cU} (H_i,K_i)$ and Toeplitz exactness (Lemma \ref{lem:toeplitz-exactness})
we get a *-homomorphism
$$
\sigma_1\colon \cS(H,K)\to \prod_\cU \cS(H_i,K_i)
$$
intertwining expectations. From weak selflessness of $(H_i,K_i)$ for all $i$, we have for each $i$  a *-homomorphism $\sigma_{i,2}\colon \cS(H_i,K_i)\to A^{\mathcal V}$ (we pick a single ultrapower for all $i$). From these embeddings we get $[(\sigma_{i,2})_i]\colon \prod_\cU \cS(H_i,K_i)\to (A^{\mathcal V})^\cU$. (We use $[\cdot]$ to denote the quotient map onto the ultrapower.) Set 
$\tilde\sigma=[(\sigma_{i,2})_i]\circ\sigma_1$.
If  $\alpha\in CB(A)$ is an extended coefficient map for $\cU$-almost all $H_i$, then  
\[
(\alpha^{\mathcal V})^{\cU}\circ \tilde \sigma=\Delta_A\circ \alpha \circ E,
\]
where $E\colon \cS(H,K)\to A$ is the vacuum expectation.
By assumption, this holds for all the coefficient maps of $H$. 
Thus, $(H,K)$ is selfless.
\end{proof}

We have a version of the previous theorem for 
Toeplitz selflessness. The proof is identical to that of Theorem \ref{thm:weak-containment-permanence-selfless}, using the Toeplitz part of Lemma \ref{lem:toeplitz-exactness} and replacing semicircular algebras by Toeplitz–Pimsner algebras.

\begin{theorem}\label{thm:weak-containment-permanence-toeplitz}
Suppose that $H_1\precsim H_2$.
\begin{enumerate}[(i)]
\item If $H_2$ is weakly Toeplitz selfless, then so is $H_1$.
\item If $H_2$ is Toeplitz selfless, and  moreover for each $v,w\in H_1$,   $\kappa_{v,w}$ belongs to the extended coefficients
set of $H_2$,  then $H_1$ is Toeplitz selfless.	
\end{enumerate}
\end{theorem}

\section{Selfless cp maps on $C^*$-algebras}\label{sec: cp maps}
Let $A$ be a unital C*-algebra and  $\phi\colon A \to A$ a cp map. Denote by $A\otimes_\phi A$ the C*-correspondence over $A$ obtained from $\phi$ via the KSGNS construction \cite{Lance1995}. We note that $A\otimes_{\phi} A$ is singly generated by $\xi_\phi=1\otimes 1$ as closed bimodule.

We endow $A\otimes_\phi A$ with the real structure generated by the vector $\xi_\phi$, which we call  the \emph{standard real structure}, and denote by $\std$.

\begin{definition}\label{def:selfless-cp-map}
We call a cp map $\phi\colon A\to A$ (weakly) selfless if  $(A\otimes_\phi A, \std)$ is (weakly) selfless. We call $\phi$ (weakly) Toeplitz selfless if $A\otimes_\phi A$ is (weakly) Toeplitz selfless.
\end{definition}

Let $\ell_\phi\in \cT(A\otimes_\phi A)$ be the left creation operator associated to $\xi_\phi$, and set $s_\phi = \ell_\phi+\ell_\phi^*$. Then $C^*(A,\ell_\phi)=\cT(A\otimes_\phi A)$ and $C^*(A,s_\phi)=\cS(A\otimes_\phi A ,\std)$. 
Let $E\colon C^*(A,\ell_\phi)\to A$ be the vacuum expectation. The  coefficient map $\kappa_{\xi_\phi,\xi_\phi}$ associated to $\xi_\phi$ is precisely $\phi$:
$$
\kappa_{\xi_\phi,\xi_\phi}(a) = \langle 1\otimes 1,\ a(1\otimes 1)\rangle=\phi(a),\qquad a\in A.
$$
Since $A\otimes_\phi A$ is singly generated by $\xi_\phi$, and
$\kappa_{\xi_\phi,\xi_\phi}=\phi$, we deduce from
Lemma \ref{lem:selfless-check-generators} that $\phi$ is selfless if and only if  the embedding
$$
(A;\phi)\hookrightarrow (C^*(A,s_\phi);\phi\circ E).
$$
is positively existential. Similarly, $\phi$ is Toeplitz selfless if and only if 
$$
(A;\phi)\hookrightarrow (C^*(A,\ell_\phi);\phi\circ E).
$$
is positively existential.

\begin{definition}\label{def:asa-inner-cp-map}
Let $\phi\in CP(A)$. We call $\phi$ approximately inner if it is a point-norm limit of cp maps of the form
$$
x\mapsto \sum_{i=1}^n v_i^*xv_i,
$$
where $v_1,\ldots,v_n\in A$. 
We call $\phi$ approximately selfadjoint-inner (a.s.a. inner) if it is a point-norm limit of cp maps of the form
$$
x\mapsto \sum_{i=1}^n h_i x h_i,
$$
where $h_1,\ldots,h_n\in A_{sa}$.
\end{definition}

If $\phi$ is weakly selfless, then $(A\otimes_\phi A, \std)\precsim (A,\std)$, by Theorem \ref{thm:asainner-correspondence}. This readily implies that $\phi$
is a.s.a. inner. In fact, a.s.a. innerness of $\phi$ holds in a uniform way, in the sense of the following theorem:

\begin{theorem}\label{thm:weak-selfless-uniform-asa-inner}
If $\phi\in CP(A)$ is weakly selfless, then it is a.s.a. inner. 
Moreover, for every $\epsilon>0$ there exists $N\in\N$ such that, for every finite set $F\subseteq A$, there exist selfadjoint elements $h_1,\ldots,h_N\in A$ satisfying
$$
\Big\|\phi(a)-\sum_{i=1}^N h_i a h_i\Big\|\leq \epsilon\|a\|	
$$
for all $a\in F$. Here $N$ depends only on $\epsilon$, and not on $F$.
\end{theorem}

\begin{proof} 
This follows from the proof of Theorem \ref{thm:asainner-correspondence}
applied to the singleton set $\{v\}=\{\xi_\phi\}\subset K$. The number $N$
chosen in the application of Voiculescu's inequality depends on $\epsilon$, but not on the elements of the finite set $F$.  
\end{proof}

Given a cp map $\phi\in CP(A)$, let us call $\psi\in CP(A)$ an extended coefficient
of $\phi$ if it is an operator norm limit of cb maps of the form
$$
x\mapsto\sum_{i=1}^n (\beta_i\circ \phi)(a_ixb_i),
$$
where $a_i,b_i\in A$ and $\beta_i\in CB(A)$. Note that in particular if
$\psi=\beta\circ\phi$ for some $\beta$, then $\psi$ is an extended coefficient
of $\phi$.

\begin{theorem}\label{thm:cp-map-permanence}
Let $A$ be a unital C*-algebra.		
\begin{enumerate}[(i)]
\item
If $\phi\in CP(A)$ is weakly selfless, then so is 
$$
\psi: x\mapsto \sum_{k=1}^n\sum_{i,j=1}^{m_k} a_{k,i}\phi(a_{k,i}^*xa_{k,j})a_{k,j}^*,
$$
where $a_{k,i}\in A$.  If $\phi$ is selfless, then the displayed map is selfless as well.

\item	
If $(\phi_i)_i\in CP(A)$ is a net of weakly selfless cp maps, and $\phi_i\to \phi$ in 
point-norm, then $\phi$ is weakly selfless. If, in addition, each $\phi_i$ is selfless and $\phi$ is an extended coefficient of $\phi_i$ for all $i$, then $\phi$ is selfless.

\item If $\phi,\psi\in CP(A)$  satisfy that $\phi$ is weakly selfless and $\psi$ is 
a.s.a. inner, then
$$
\theta:=\frac12(\phi\psi+\psi\phi),
$$
is weakly selfless. If in addition  $\phi$ is selfless, and $\theta$ is an extended coefficient of $\phi$, then  $\theta$ is selfless. In particular, this holds if $\psi$ and $\phi$ commute.

\item 
If $\phi\in CP(A_1)$ is weakly selfless, $\psi\in CP(A_2)$ is a.s.a. inner,
and $A_2$ is exact, then $\phi\otimes\psi$ is weakly selfless. If, moreover, $\phi$ is selfless, then $\phi\otimes\psi$ is selfless.
(The tensor product is the minimal one.)

\item 
If $\phi$ is a selfless unital cp map and $p\in A$ is a projection fixed by $\phi$, then $\phi|_{pAp}\colon pAp\to pAp$ is selfless.
	\end{enumerate}
\end{theorem}

\begin{proof}
(i) For $j=1,\ldots,n$, define $v_j\in A\otimes_\phi A$ by
$$
v_k=\sum_{i=1}^{m_k} a_{k,i}\xi_{\phi} a_{k,i}^*.  
$$
Now, using Lemma \ref{lem:singly-generated-embedding}, embed $(A\otimes_\psi A,\std)$ into $(A\otimes_\phi A,\std)^{\oplus n}$ by  
mapping 
$$
\xi_{\psi}\mapsto (v_1,\ldots,v_n).
$$
Apply Theorem \ref{thm:weak-containment-permanence-selfless}.

(ii) 
Choose an ultrafilter $\cU$ on the index set and containing the tail sets $\{i:i\geq i_0\}$.
From the point-norm convergence $\phi_i\to \phi$, we get a covariant embedding
\begin{align*}
(A\otimes_\phi A,\std)&\to \prod_{\cU} (A\otimes_{\phi_i} A,\std)\\
\xi_\phi &\mapsto [(\xi_{\phi_i})_i], 
\end{align*}
where the underlying *-homomorphism is $\Delta_A\colon A\to A^{\cU}$. We now apply Theorem \ref{thm:ultraproduct-permanence-selfless}.

(iii) Note first that for selfadjoint $a\in A$, the vectors 
$$
p_a=\frac{1}{2}(\xi_\phi a+a\xi_\phi),
\qquad
q_a=\frac{1}{2i}(\xi_\phi a-a\xi_\phi).
$$
belongs to  $\std$ (are real) and satisfy that
$$
\kappa_{p_a,p_a}(x)+\kappa_{q_a,q_a}(x)
=
\frac12(a\phi(x)a+\phi(axa)).
$$

Suppose that $\psi$ is selfadjoint inner, i.e., 
$$
\psi(x)=\sum_{j=1}^n a_jxa_j,
$$
for some $a_1,\ldots,a_n$ selfadjoint. Define $p_j,q_j\in A\otimes_{\phi} A$ by
$$
p_j=\frac{1}{2}(\xi_\phi a_j+a_j\xi_\phi),
\qquad
q_j=\frac{1}{2i}(\xi_\phi a_j-a_j\xi_\phi).
$$
Then $(A\otimes_\theta A,\std)$ embeds into $(A\otimes_\phi A,\std)^{\oplus 2n}$ via  
mapping 
$$
\xi_{\theta}\mapsto (p_1,q_1,\ldots,p_n,q_n).
$$
(Thus, $\theta$ has  the form dealt with in (i)). If $\psi$ is a.s.a. inner, we deduce that $A\otimes_\theta A$ is weakly contained in $A\otimes_\phi A$. Apply Theorem \ref{thm:weak-containment-permanence-selfless}.

(iv) We first show that $\phi\otimes \mathrm{id}_{A_2}$ is weakly selfless whenever $\phi$ is weakly selfless, and selfless whenever $\phi$ is selfless. We will denote the C*-correspondence obtained from a cp map $\phi$ by $H_\phi$, rather than $A\otimes_\phi A$ as we have done before, to keep the notation compact.

Let $C=\cS(\ell^2(H_\phi),\ell^2(\std))$, and let $E\colon C\to A_1$ be the vacuum expectation. Since $\phi$ is weakly selfless, there exist an ultrafilter $\cU$ and a *-homomorphism $\sigma\colon C\to A_1^{\cU}$
such that $\sigma|_{A_1}=\Delta_{A_1}$. If $\phi$ is selfless, then, by the $\ell^2$-stability of selflessness, $\sigma$ may moreover be chosen so that
$$
\phi^{\cU}\circ\sigma = \Delta_{A_1}\circ\phi\circ E.
$$
Tensoring $\sigma$ with $\id_{A_2}$ gives a *-homomorphism
$$
\sigma\otimes\id_{A_2}\colon C\otimes A_2\to A_1^{\cU}\otimes A_2.
$$
Since $A_2$ is exact, there is a  canonical embedding
$A_1^{\cU}\otimes A_2\to(A_1\otimes A_2)^{\cU}$.
We therefore obtain a *-homomorphism
$\widetilde{\sigma}\colon C\otimes A_2\to(A_1\otimes A_2)^{\cU}$.

There is a natural isomorphism
$$
\cS(\ell^2(H_\phi),\ell^2(\std))\otimes A_2
\cong \cS(\ell^2( H_{\phi\otimes\id_{A_2}}),
\ell^2(\std)),
$$
under which the expectation $E\otimes\id_{A_2}$ on the left-hand side  corresponds to the vacuum expectation on the right-hand side. (This is  proven by first obtaining the corresponding isomorphism for Toeplitz-Pimsner C*-algebras using their universal property, and then observing that it  restricts to the isomorphism displayed above.) Since $\widetilde{\sigma}$ restricts to the diagonal embedding on $A_1\otimes A_2$, this proves that $\phi\otimes\id_{A_2}$ is weakly selfless. In the selfless case, the canonical embedding $A_1^{\cU}\otimes A_2\to (A_1\otimes A_2)^{\cU}$ intertwines $\phi^{\cU}\otimes\id_{A_2}$ with
$(\phi\otimes\id_{A_2})^{\cU}$. Consequently,
$$
(\phi\otimes\id_{A_2})^{\cU}\circ\widetilde{\sigma}
=
\Delta_{A_1\otimes A_2}
\circ(\phi\otimes\id_{A_2})
\circ(E\otimes\id_{A_2}).
$$
Thus $\phi\otimes\id_{A_2}$ is selfless.

Since $\psi$ is a.s.a. inner, so is $\mathrm{id}_{A_1}\otimes \psi$. Moreover, $\phi\otimes \mathrm{id}_{A_2}$ and
$\mathrm{id}_{A_1}\otimes \psi$ commute, and
$$
\phi\otimes \psi =
(\phi\otimes \mathrm{id}_{A_2})\circ
(\mathrm{id}_{A_1}\otimes \psi).
$$
The result follows from part (iii).

(v)    We have
$$
C^*(pAp, s_{\phi|_{pAp}}) \subset p C^*(A, s_\phi)p \subset (pAp)^{\cU}.
$$

\end{proof}

\begin{remark}
The previous theorem holds with ``completely selfless'' in place of ``selfless'', and in this case we may remove the assumption
of exactness of $A_2$ in (iv).

\end{remark}

\begin{corollary}\label{cor:cp-map-powers-and-cesaro-limits}
If $\phi\in CP(A)$ is (weakly) selfless, then so is $\phi^k$ for all $k=1,2,\ldots$. Moreover, if the limit
$$
\frac1N\sum_{k=1}^N \phi^k\to P
$$
exists in point-norm, then $P$ is (weakly) selfless.
\end{corollary}
\begin{proof}
Let $k\geq 1$. Suppose, inductively, that $\phi^k$ is (weakly) selfless (the case $k=1$ is our hypothesis). Since $\phi$
is a.s.a. inner, and $\phi$ and $\phi^k$ commute,  
$\phi^{k+1}=\phi^{k}\circ\phi$ is (weakly)
selfless by Theorem \ref{thm:cp-map-permanence} (iii).

Suppose that the displayed point norm limit exists. Then $P$ is a.s.a. inner, as it is a point-norm limit of a.s.a. inner maps,  and $\phi P=P\phi =P$. By Theorem \ref{thm:cp-map-permanence} (iii), $P$ is (weakly) selfless.
\end{proof}

\begin{corollary}
    Let $G$ be a countable discrete group.  Let $v \in \ell^2(G; \R)$, and let $\phi_{v}\colon C_r^*(G)\to C_r^*(G)$ be the Fourier multiplier map associated to 
    $g\mapsto \langle v, \lambda_g v\rangle$, i.e.,  
    $$
    \phi_v(\lambda_g)=\langle v, \lambda_g v\rangle\lambda_g,
    $$
    for all $g\in G$. If $(C_r^\ast(G),\tau)$ is (completely) selfless (where $\tau$ denotes the canonical trace on $C_r^*(G)$), then $\phi_v$ is (completely) selfless.
\end{corollary}

\begin{proof}
The assumption of (complete) selflessness of $(C_r^\ast(G),\tau)$ means that the trace $\tau$
is (completely) selfless when viewed as a cp map on $C_r^*(G)$ (see Theorem \ref{thm:prob-selfless-equivalent-cp-map}). We can embed 
$(C_r^\ast(G) \otimes_{\phi_v} C_r^\ast(G), \std)$ in 
$(C_r^\ast(G) \otimes_\tau C_r^\ast(G), \std)$ by  the assignment
    \begin{equation*}
   \xi_{\phi_v}\mapsto     \sum_{g \in G} v_g (\lambda_g \otimes \lambda_g^\ast).
    \end{equation*}
    (Indeed, a quick calculation shows that the vector on the right-hand side gives rise to $\phi_v$ as a coefficient map.)
The result now follows from Theorem \ref{thm:selfless-subcorrespondences-directed-unions}.    
\end{proof}

The following is a cp-map analogue of Ozawa's selflessness criterion \cite[Lemma 12 and Theorem 13]{ozawa2025proximalityselflessnessgroupcalgebras} for the case of a state.

\begin{proposition}\label{prop:cp-map-selflessness-criterion}
Let $\phi\in CP(A)$. Suppose that there exist a unital inclusion $A\subseteq C$, a cp map $\psi\in CP(C)$ extending $\phi$ 
to $C$,  and   $v\in C^{\cU}$ (for some ultrafilter) such that 
\begin{enumerate}[(1)]
\item
$v^*av=\phi(a)$ for all  $a\in A$,
\item 
$\psi^{\cU}(avv^*a^*)=0$ for all $a\in A$,
\item
$v+v^*\in A^{\cU}$.
\end{enumerate}
Then $\phi$ is selfless.
\end{proposition}

\begin{proof}
From $v^*av=\phi(a)$ we get a representation of $A\otimes_\phi A$ in $C^{\cU}$ 
(a covariant embedding into the identity correspondence of $C^{\cU}$) such that
$\xi_{\phi}\mapsto v$, and with underlying *-homomorphism the embedding of $A$ in $C^{\cU}$ (Lemma \ref{lem:singly-generated-embedding}). Thus, by the universal property of the Toeplitz-Pimsner C*-algebra, there exists a *-homomorphism $\sigma\colon \cT(A\otimes_{\phi} A)\to C^{\cU}$ such that $\sigma|_A$ is the diagonal embedding of $A$ in $C^{\cU}$ and
$\sigma(\ell_{\phi})=v$. 
Since $v+v^*\in A^{\cU}$, the restriction of $\sigma$ 
to $C^*(A,s_{\phi})$ ranges in $A^{\cU}$.

Consider $\tilde\phi=\psi^{\cU}\sigma\colon \cT(A\otimes_{\phi}A)\to C^{\cU}$.
From $\psi^{\cU}(avv^*a^*)=0$ for all $a\in A$ we deduce that $\tilde\phi(a\ell_{\phi}\ell_{\phi}^*a^*)=0$ for all $a\in A$. By Cauchy-Schwarz, we obtain that
$$
\tilde\phi(a\ell_{\phi}x)=\tilde\phi(x\ell_{\phi}^*a)=0
$$
for all $a\in A$ and $x\in \cT(A\otimes_\phi A)$. Together with the fact that $\tilde\phi|_A=\Delta_A\circ \phi|_A$, this implies  that $\tilde\phi=\Delta_A\circ \phi\circ E$, i.e.,
 $\psi^{\cU}\circ \sigma=\Delta_A\circ \phi\circ E$, where $E\colon \cT(A\otimes_\phi A)\to A$ denotes the vacuum expectation. Since $\psi$ extends $\phi$ and  $\sigma|_{C^*(A,s_\phi)}$ ranges in $A^{\cU}$, we get
 $\phi^{\cU}\circ \sigma|_{C^*(A,s_{\phi})}=\Delta_A\circ \phi\circ E$, as desired.
\end{proof}

We formulate a finite sets and $\epsilon$ version of the above criterion:

\begin{proposition}\label{prop:cp-map-selflessness-criterion-epsilon}
Let $\phi\in CP(A)$. Suppose that there exist a unital inclusion $A\subseteq C$ and a cp map $\psi\in CP(C)$ extending $\phi$ 
to $C$ such that for every finite set $F\subset A$
and $\epsilon>0$, there exists $v\in C$ such that 
\begin{enumerate}[(1)]
\item
$\|v^*av-\phi(a)\|<\epsilon$ for all  $a\in F$,
\item 
$\|\psi(avv^*a^*)\|<\epsilon$ for all $a\in F$,
\item
$\|v+v^*-s\|<\epsilon$ for some  $s\in A$.
\end{enumerate}
Then $\phi$ is selfless.
\end{proposition}

We finish this section with variants of the previous results for (weak) Toeplitz selflessness. The arguments run along similar, but  simpler lines.
\begin{proposition}
If $\phi$ is weakly Toeplitz selfless, then it is 1-step approximately inner.
That is,  $\phi$ is a point-norm limit of cp maps maps of the form
$x\mapsto v^*xv$.
\end{proposition}

\begin{proof}
Let $\xi_1$ be the first
copy of $\xi_\phi$ in $\ell^2(A\otimes_\phi A)$, and let $\ell_1$ be the
corresponding creation operator in $\cT(\ell^2(A\otimes_\phi A))$. Then $\ell_1^*a\ell_1=\phi(a)$ for all $a\in A$. Applying 
$\sigma\colon \cT(\ell^2(A\otimes_\phi A))\to A^\cU$
witnessing weak Toeplitz selflessness, we get
\[
\sigma(\ell_1)^*a\sigma(\ell_1)=\phi(a)
\]
for all $a\in A$. Thus, $\phi$ is approximately 1-step inner. 
\end{proof}

We have a version of Theorem \ref{thm:cp-map-permanence} for Toeplitz selflessness:

\begin{theorem}\label{thm:toeplitz-cp-map-permanence}
Let $A$ be a unital C*-algebra.		
\begin{enumerate}[(i)]
\item
If $\phi\in CP(A)$ is (completely, weakly) Toeplitz selfless, then so is 
$$
\psi:x\mapsto \sum_{k=1}^n\sum_{i,j=1}^{m_k} b_{k,i}\phi(a_{k,i}^*xa_{k,j})b_{k,j}^*,
$$
where $a_{k,i},b_{k,i}\in A$.

\item	
If $(\phi_i)_i\in CP(A)$ is a net of weakly Toeplitz selfless cp maps, and $\phi_i\to \phi$ in  point-norm, then $\phi$ is weakly Toeplitz selfless. If, in addition, each $\phi_i$ is Toeplitz selfless and $\phi$ is an extended coefficient of $\phi_i$ for all $i$, then $\phi$ is Toeplitz selfless.

\item If $\phi,\psi\in CP(A)$  satisfy that $\phi$ is (weakly) Toeplitz selfless and $\psi$ is  approximately inner, then $\psi\phi$ is (weakly) Toeplitz selfless. 

\item 
If $\phi\in CP(A_1)$ is weakly Toeplitz selfless and $\psi\in CP(A_2)$ is approximately inner,  then $\phi\otimes\psi$ is weakly Toeplitz selfless. If, moreover, $\phi$ is Toeplitz selfless, then $\phi\otimes\psi$ is Toeplitz selfless.

\item 
If $\phi$ is unital and Toeplitz selfless, and $p\in A$ is a projection fixed by $\phi$, then $\phi|_{pAp}\colon pAp\to pAp$ is Toeplitz selfless.
	\end{enumerate}
\end{theorem}

In the case of Toeplitz selflessness,   the criterion  
Proposition \ref{prop:cp-map-selflessness-criterion} takes on a much simpler form: 
\begin{proposition}
A cp map $\phi\in CP(A)$ is Toeplitz selfless if and only if there exists  $v\in A^{\cU}$ such that 
$v^*av=\phi(a)$ and $\phi^{\cU}(avv^*a^*)=0$ for all $a\in A$.
\end{proposition}

\section{Some regularity properties}\label{sec: regularity}
In this section we explore some regularity properties of $A$ when  $(A,\phi)$ is weakly selfless. The situation is most satisfactory in the case of Toeplitz selflessness, so we deal with this case first.

In \cite{KirRorNonsimpleOinfty} Kirchberg and R{\o}rdam introduce the notion of strongly purely infinite C*-algebra.
A C*-algebra $A$
is strongly purely infinite if for all $a_1,a_2\in A_+$  and 
$\epsilon>0$ there exist $d_1,d_2\in A$ such that 
$$
\|d_1^*a_1d_1-a_1\|<\epsilon, \quad \|d_2^*a_2d_2-a_2\|<\epsilon, \quad
\|d_2^*a_2^{\frac12}a_1^{\frac12}d_1\|<\epsilon.
$$
This is a strengthening of the notion of purely infinite C*-algebra in the non-simple case. By \cite[Theorem 4.16]{KirRorNonsimpleOinfty}, if $A$ is separable and nuclear,
then it is strongly purely infinite if and only if
$A\otimes \mathcal O_\infty\cong A$.

\begin{theorem}
If $(A,\phi)$ is weakly Toeplitz selfless and $\mathrm{Ideal}(\phi(a))=\mathrm{Ideal}(a)$ for all  $a\in A_+$,  then $A$ is strongly purely infinite.
\end{theorem}

\begin{proof}
We start working in 
$
C=\cT(\ell^2(A\otimes_\phi A))
$
then pass to $A^{\cU}$ via $\sigma\colon C\to A^{\cU}$ such that $\sigma|_A=\Delta_A$. 

Let us first show that $a\precsim \phi(a)$, in the Cuntz preorder of $C$, for all $a\in A_+$. Let $a\in A_+$ and $\epsilon>0$. 
By assumption,   $a$ belongs to the ideal generated by $\phi^{2}(a)$. Thus,
$$
\Big\|a-\sum_{i=1}^n x_i^*\phi^{2}(a)x_i\Big\|<\epsilon,
$$
for some $x_i\in A$. Now choose left creation operators  $\ell_1,\ldots,\ell_n\in C$
corresponding to pairwise orthogonal copies $\xi_\phi^{(1)},\ldots, \xi_\phi^{(n)}$
of the canonical generator $\xi_\phi\in A\otimes_\phi A$.  Set $x=\sum_{i=1}^n \ell_ix_i$. Then 
$$
x^*\phi(a)x=\sum_{i=1}^n x_i^*\ell_i^*\phi(a)\ell_ix_i=\sum_{i=1}^n x_i^*\phi^2(a)x_i, 
$$
where we have used that $\ell_i^*b\ell_j=0$ if $i\neq j$ and $b\in A$. Thus,
$\|a-x^*\phi(a)x\|<\epsilon$. This shows that $a\precsim \phi(a)$.

Let $a_1,a_2\in A_+$ and $\epsilon>0$.  Choose left creation operators  $\ell_1,\ell_2\in C$ as before. Then,
$$
\ell_1^*a_1\ell_1 =\phi(a_1),\quad
\ell_2^*a_2\ell_2 =\phi(a_2),\quad
\ell_2^*a_2^{\frac12}a_1^{\frac12}\ell_1=0.
$$
Since $a_1\precsim \phi(a_1)$, we can find $x_1\in C$ such that 
$\|a_1-x_1^*\phi(a_1)x_1\|<\epsilon$, and similarly we find $x_2\in C$ such that $\|a_2-x_2^*\phi(a_2)x_2\|<\epsilon$. 
Set $d_1=\ell_1x_1$ and $d_2=\ell_2x_2$. Then 
$$
\|d_1^*a_1d_1-a_1\|<\epsilon, \quad \|d_2^*a_2d_2-a_2\|<\epsilon, \quad
d_2^*a_2^{\frac12}a_1^{\frac12}d_1=0.
$$
Now we pass to $A^{\cU}$ via $\sigma$, and then to $A$ by projecting onto a suitable coordinate.
\end{proof}

The next theorem shows that if $(A,\phi)$ is weakly selfless, 
then the set $\overline{[A,A]}$ is definable in the model theory of $A$;
see \cite{FHLRTVW2021modelthryCstar}.
It also shows that $A^{\cU}$ has no ``phantom'' traces. More concretely, for every  tracial state $\tau$ on $A^{\cU}$ and separable C*-subalgebra $D\subseteq A^{\cU}$, there exist tracial states $(\tau_i)_i$ on $A$ such that  $\tau$ agrees with the limit trace $\lim_{\cU}\tau_i$ on $D$; 
see \cite[Theorem 8]{Ozawa2013Dixmier} and \cite{AntoinePereraRobertThiel2024}.

\begin{theorem}\label{thm:no-phantom}
If $(A,\phi)$ is weakly selfless and $\phi$ is unital, then  $\overline{[A^{\cU},A^{\cU}]}=\overline{[A,A]}^{\cU}$.
\end{theorem}

\begin{proof}
It is enough to show that for every $\gamma>0$ there exist $N\in\N$ and $M>0$ such that, for every $a\in\overline{[A,A]}$, there exist $x_i,y_i\in A$, $i=1,\ldots,N$, satisfying
\begin{equation}\label{approxcomm}
	\Big\|a-\sum_{i=1}^N[x_i,y_i]\Big\|\leq\gamma\|a\|,
	\qquad
	\|x_i\|\cdot \|y_i\|\leq M\|a\|.
\end{equation}
Here $N$ and $M$ may depend on $\gamma$, but not on $a$. See \cite[Theorem A]{AntoinePereraRobertThiel2024}.
We will prove the required approximation for $a\in\overline{[A,A]}\cap A_{\mathrm{sa}}$. For a general $a\in\overline{[A,A]}$, write
$$
a=\operatorname{Re}(a)+i\operatorname{Im}(a).
$$
Both terms belong to $\overline{[A,A]}\cap A_{\mathrm{sa}}$ and have norm at most $\|a\|$. Applying the selfadjoint case to each term with error factor $\gamma/2$ and adding the resulting approximations gives \eqref{approxcomm} for $a$.

Let $\gamma>0$ and $a\in \overline{[A,A]}\cap A_{\mathrm{sa}}$.
By the uniform form of the a.s.a. inner property of $\phi$ from Theorem \ref{thm:weak-selfless-uniform-asa-inner} applied with $\epsilon=\gamma/2$ and  $F=\{1,a\}$,
there exist selfadjoint elements $x_1,\ldots,x_N\in A$ satisfying
$$
\Big\|\phi(a)-\sum_{i=1}^{N}x_iax_i\Big\|\leq \frac{\gamma}{2}\|a\|, \qquad \Big \|1 -\sum_{i=1}^N x_i^2\Big\| \leq \frac{\gamma}{2}.
$$
Here $N$ depends on $\gamma$ but is independent of $a$. 
Since
\begin{align*}
a-\phi(a)=a-\sum_{i=1}^{N} x_i^2a + \sum_{i=1}^{N} [x_i,ax_i] + \sum_{i=1}^{N} x_iax_i-\phi(a),
\end{align*}
we have
$$
\Big\| (a- \phi(a)) - \sum_{i=1}^N [x_i,ax_i]\Big\|\leq \Big\|1-\sum_{i=1}^N x_i^2\Big\|\cdot \|a\| + \frac{\gamma}{2}\|a\|
\leq \gamma\|a\|.
$$
Hence,
$$
\Big\| \frac12(a- \phi(a)) - \frac12\sum_{i=1}^N [x_i,ax_i]\Big\|\leq \frac{\gamma}{2}\|a\|.
$$
Notice that $\|x_i\|\cdot \|ax_i\|\leq \|x_i\|^2\cdot \|a\|\leq (1+\gamma/2)\|a\|$.

We will be done once we have shown how to approximate
$\frac12(a+\phi(a))$ by a uniformly bounded number of commutators
with error at most $\frac{\gamma}{2}\|a\|$.
Adding the approximations of $\frac12(a- \phi(a))$ and $\frac12(a+\phi(a))$ then gives \eqref{approxcomm}.

The next part of the proof is inspired by the proof of  \cite[Theorem 6]{Ozawa2013Dixmier} (but we use freeness in place of commutativity in some key places). By \cite[Theorem 5]{Ozawa2013Dixmier}, there exists a universal constant $M_1>0$ such that elements of the form
$$
a=\sum_{i=1}^n [x_i^*,x_i],
$$
where $x_i\in A$ satisfy that   $\sum_{i=1}^n \|x_i\|^2\leq M_1\|a\|$, form a  dense set in $\overline{[A,A]}\cap A_{sa}$. 
Since $\frac12(\mathrm{id}_A+\phi)$ is contractive, density shows that
it is enough to prove the required approximation of  $\frac12(a+\phi(a))$ for elements $a$ of the form
$$
a=\sum_{i=1}^n[x_i^*,x_i],
\qquad
\sum_{i=1}^n\|x_i\|^2\leq M_1\|a\|.
$$
Let us assume that $a$ has this form.

We now work in the C*-algebra
$$
C=\cS(\ell^2(A\otimes_\phi A), \ell^2(\std)),
$$
and let $E\colon C\to A$ denote the expectation onto $A$. 
Choose a doubly indexed collection of semicircular elements  $s_{ij}\in C$ 
associated to pairwise orthogonal copies of the canonical generator $\xi_\phi\in A\otimes_\phi A$,  where $i=1,\ldots,n$ and $j=1,\ldots,R$ ($R$ is chosen below). For $j=1,\ldots,R$, set
$$
X_j=\sum_{i=1}^n s_{i,j}x_i,
\qquad
\widetilde X_j=\sum_{i=1}^n x_i s_{i,j}.
$$
Observe that $E(s_{ij}x_i)=0$. Hence, $E(X_j)=0$ and similarly, $E(\widetilde X_j)=0$. 
We can estimate the  norms  
of $X_j,\tilde X_j$ using Voiculescu's inequality:
\begin{align*}
\|X_j\| \leq \sup_i \|s_{ij}x_i\|+\Big\|\sum_{i=1}^n x_i^*x_i\Big\|^{\frac12}
+\Big\|\sum_{i=1}^n x_ix_i^*\Big\|^{\frac12}\leq M_2\|a\|^{\frac12},
\end{align*}
for a universal constant $M_2$.
Similarly, $\|\widetilde X_j\|\leq M_2\|a\|^{\frac12}$.

We have
$$
[X_j,X_j^*]=\sum_{i=1}^n [s_{i,j}x_i,x_i^*s_{i,j}]+\sum_{i\neq i'} [s_{i,j}x_i,x_{i'}^*s_{i',j}].
$$
Let us apply $E$ on both sides. On the right-hand side, we use  
that the terms of the second sum have zero expectation and that
$$
E([s_{ij}x_i,x_i^*s_{i,j}])=E(s_{ij}x_ix_i^*s_{ij})-E(x_i^*(s_{ij})^2x_i)=\phi(x_ix_i^*)-x_i^*x_i.
$$ 
This yields
$$
E([X_j,X_j^*])=\sum_{i=1}^n \phi(x_ix_i^*)- x_i^*x_i.
$$
A similar calculation shows that
$$
E([\widetilde X_j,\widetilde X_j^*])=\sum_{i=1}^n x_ix_i^*- \phi(x_i^*x_i).
$$
We deduce that
$$
E([X_j^*,X_j]+[\widetilde X_j^*,\widetilde X_j])= a+\phi(a).
$$
As $j$ varies, the centered elements 
$$
[X_j^*,X_j]+[\widetilde X_j^*,\widetilde X_j] - (a+\phi(a))
$$
belong to freely independent factors of $C$. Applying Voiculescu's inequality,
we get
\begin{align*}
\Big\|
\frac1R\sum_{j=1}^R ([X_j^*,X_j]+[\widetilde X_j^*,\widetilde X_j]) - (a+\phi(a))
\Big\|
&\leq 
\Big(\frac2R + \frac2{\sqrt R}\Big)\sup_j \|[X_j^*,X_j]+[\widetilde X_j^*,\widetilde X_j]\|\\
&\leq \Big(\frac2R + \frac2{\sqrt R}\Big)M_3\|a\|.
\end{align*}
Choosing $R$ sufficiently large, and absorbing the coefficient
$1/(2R)$ into the entries of each commutator, gives an approximation
of the form \eqref{approxcomm} for $\frac12(a+\phi(a))$, with constants
independent of $a$.

To finish the proof, we use that by  weak selflessness of $\phi$ there exists
a *-homomorphism $\sigma\colon C\to A^{\cU}$ that restricts to the diagonal inclusion on $A$.
Applying $\sigma$ gives the required approximation in $A^\cU$.
Lifting the finitely many entries of the commutators and passing to a suitable coordinate gives the corresponding approximation in $A$.
\end{proof}

\section{Selfless operator-valued $C^*$-probability spaces}\label{sec: expectations}

Let $A$ be a unital C*-algebra with unital C*-subalgebra $B$ and  expectation $E\colon A\to B$. We call $(A,E,B)$ a C*-probability space. We assume that the GNS representation of $A$ in $L^2(A,E)$ induced by $E$ is faithful (except for ultrapowers/ultraproducts); this ensures that the canonical maps into reduced amalgamated free products are faithful. By a morphism between C*-probability spaces we understand a unital *-homomorphism between the C*-algebras that intertwines the expectations.  

We use the reduced amalgamated free product construction of $C^*$-probability
spaces with conditional expectations; see
\cite[Chapter 4]{voiculescu1992free}.

\begin{definition}\label{def:selfless-cstar-probability-space}
We call $(A,E,B)$ selfless if for some C*-probability space $(C,\kappa)$, with $C\neq \C$ and $\kappa$ a state inducing a faithful GNS,
the first factor  embedding 
$$
A\to A*_B (B\otimes C)
$$
is an existential embedding of C*-probability spaces. More concretely, 
there exists an embedding $\sigma\colon A*_B (B\otimes C)\to A^\cU$  whose restriction to $A$ agrees with the diagonal embedding $\Delta_A$ and such that
$$
E^{\cU}\circ \sigma = \Delta_A \circ (E*_B (\mathrm{id}_B\otimes\kappa)).
$$
We call $(A,E,B)$ Toeplitz selfless if $(C,\kappa)$ may be chosen  with a non-tracial $\kappa$.
\end{definition}	

In the following lemma we denote by $C^*(A,\ell_E)$
the Toeplitz-Pimsner C*-algebra of $A\otimes_E A$, i.e., 
$\cT(A\otimes_E A)$, and by $E_{\Omega}$ its vacuum expectation.
As in the previous section, we write
$C^*(A,s_E)$ for  $\cS(A\otimes_E A,\std)$.
The isomorphisms of this lemma are well known. The first isomorphism appears in
\cite[Section 4.6, exercises]{brown2008textrm} and in \cite[Section 2.5]{BrownDykemaShlyakhtenko2002}, where it is attributed to
Shlyakhtenko \cite{Shlyakhtenko1998}. 
See also \cite[Lemma 12]{ozawa2025proximalityselflessnessgroupcalgebras}.
The second isomorphism follows from the first.

\begin{lemma}\label{lem:prob-toeplitz-semicircular-identification} 
Let $(A,E,B)$ be a C*-probability space.
Let $(\cT,\omega)$ be the Toeplitz C*-algebra endowed with the vacuum state, and denote by $v\in \cT$ the generating isometry. 
We have an isomorphism of $A$-valued C*-probability spaces:
\begin{align*}
(C^*(A,\ell_E),E_\Omega,A) \cong 
(A*_B (B\otimes \cT),\mathrm{id}_A*_B (\mathrm{id}_B\otimes \omega),A),
\end{align*}
mapping the left creation operator $\ell_{E}$ to $1\otimes v$. Restricting this isomorphism to $C^*(A,s_E)$ yields 
the isomorphism
$$
(C^*(A,s_E),E_\Omega,A) \cong 
(A*_B (B\otimes C([-2,2])),\mathrm{id}_A*_B (\mathrm{id}_B\otimes \lambda),A),
$$
where the state on $C([-2,2])$ is induced by the semicircular distribution.
\end{lemma}

\begin{theorem}\label{thm:prob-selfless-equivalent-cp-map} 
Let $(A,E,B)$ be a C*-probability space. The following are equivalent:
\begin{enumerate}[(i)]
\item
$(A,E,B)$ is selfless.

\item
The embedding $A\to A*_B (B\otimes C)$ is an existential embedding of $B$-valued C*-probability spaces, where $C$ can be any of 
$$
C=C_r^*(\mathbb F_\infty), \qquad (C([-2,2]))^{*\infty}. 
$$

\item
The expectation $E\colon A\to B\subset A$ is a selfless cp map.
\end{enumerate}
Moreover, the following are equivalent:
\begin{enumerate}
\item[(v)] 
$(A,E,B)$ is Toeplitz selfless.
\item[(vi)]
The embedding $A\to A*_B (B\otimes \mathcal O_\infty)$ is an existential embedding of $B$-valued C*-probability spaces, where $\mathcal O_\infty$ is the Cuntz algebra endowed with the vacuum state.
\item[(vii)] 
The expectation $E\colon A\to B\subset A$ is a Toeplitz selfless cp map.
\end{enumerate}
\end{theorem}	

\begin{proof}
(i) $\Rightarrow$ (ii): The argument runs along the same lines 
as the argument for the scalar case ($B=\C$) found in  
\cite[Theorem 2.6]{robertselfless}. 
Set $D=B\otimes C$, with expectation
$\mathrm{id}_B\otimes \kappa$ onto $B$. 
Let $\sigma\colon A*_B D\to A^{\cU}$ be a *-homomorphism such that $\sigma|_{A}=\Delta_A$
and intertwining the respective expectations onto $B$ and $B^{\cU}$. By Pisier's amalgamated free product
strong convergence result \cite[Theorem 7.1]{Pisier2012}, or  by the more general  \cite[Corollary 1.2]{gao-KE}, we obtain  
$$
\tilde\sigma\colon (A*_B D)*_B D \to (A*_B D)^{\cU}
$$
whose restriction to $(A*_B D)$ agrees with $\sigma$
and whose restriction to the second factor $D$ agrees with the diagonal inclusion. Post-composing with $\sigma^{\cU}$ we obtain 
$$
\sigma^{\cU}\tilde\sigma\colon A*_B D*_B D\to (A^{\cU})^{\cU}
$$
that restricts to the diagonal inclusion on $A$
and intertwines expectations. 
Repeated applications of this show that the embeddings
$A\to A*_B D^{*_B n}$ are existential embeddings of operator valued C*-probability spaces for all $n$. Note that
$$
D^{*_B n}\cong B\otimes (C,\kappa)^{*n}.
$$
Since $C\neq \C$ and $\kappa$ induces a  faithful GNS representation, for large enough $n$ we have that
 $(C_r^*(\mathbb F_\infty),\tau)$ embeds in $(C,\kappa)^{*n}$
 (\cite[Lemma 2.5]{robertselfless}). We thus conclude that
 $$
 A\to (A*_B (B\otimes C_r^*(\mathbb F_\infty)).
 $$
is an existential embedding of operator-valued C*-probability spaces.
 Since $(C([-2,2]),\lambda)^{*\infty}$ embeds in 
 $(C_r^*(\mathbb F_\infty),\tau)$, we get an existential embedding in this case too.

(ii) $\Leftrightarrow$ (iii): 	
By Lemma \ref{lem:prob-toeplitz-semicircular-identification}, $C^*(A,s_E)\cong A*_B (B\otimes C([-2,2]))$.

(iii) $\Rightarrow$ (i) is obvious.

Let us now deal with the equivalence of the last three assertions. 

(vi) $\Leftrightarrow $ (vii) follows from Lemma \ref{lem:prob-toeplitz-semicircular-identification}.

(vi) $\Rightarrow$ (v) is immediate. 

To prove that (v) implies (vi), recall that, as argued above, the embedding 
$$
A\to A*_B (B\otimes (C,\kappa)^{*n}))
$$
is existential. For large enough $n$, $(C',\kappa')=C_r^*(\mathbb F_\infty)*(C,\kappa)$ embeds in $(C,\kappa)^{*n}$. Thus, the embedding
$$
A\to A*_B (B\otimes (C',\kappa'))
$$
is existential. The C*-probability
space $(C',\kappa')=C_r^*(\mathbb F_\infty)*(C,\kappa)$ is both selfless and nontracial. By Gould's dichotomy \cite{gould2026selfless}, it is purely infinite.  By Ozawa's proof that purely infinite C*-algebras are completely selfless, 
\cite[Theorem 3]{ozawa2025proximalityselflessnessgroupcalgebras},
we have a \emph{completely} existential embedding  
$$
(C',\kappa')\to (C',\kappa')*(\mathcal O_\infty,\omega).
$$
Taking tensor product with $B$, yields a (completely) existential embedding
$$
B\otimes (C',\kappa')\to (B\otimes (C',\kappa'))*_B(B\otimes \mathcal O_\infty,\omega).
$$
Taking amalgamated reduced free product with $A$ and using free exactness, we get an existential embedding
$$
A*_B B\otimes (C',\kappa')\to A*_B (B\otimes (C',\kappa'))*_B(B\otimes \mathcal O_\infty,\omega).
$$
So the embedding
$A\to A*_B (B\otimes \mathcal O_\infty,\omega)$ is also existential.
\end{proof}	
%\todo{It seems more transparent to directly work with the maps sigma into ultrapowers.}

\begin{definition}
Let us call $(A,E,B)$ weakly selfless, if $E$
is weakly selfless. Equivalently, if the embedding
$$
A\mapsto A*_B (B\otimes (C([-2,2]))^{*\infty})
$$
is an existential embedding of C*-algebras. 

We call $(A,E,B)$ weakly Toeplitz selfless 
if $E$ is weakly Toeplitz selfless. Equivalently, if the embedding 
$$
A\mapsto A*_B (B\otimes (\mathcal O_\infty,\omega))
$$
is an existential embedding of C*-algebras. 
\end{definition}

\begin{remark}
Hirshberg and Phillips show in \cite[Corollary 3.5]{hirshberg-phillips} that
$$
C([-2,2])^{*\infty}\cong\mathcal Z^{*\infty},
$$
where the free products are taken with respect to the semicircular
state on $C([-2,2])$ and the unique tracial state on $\mathcal Z$.
Thus, $\mathcal Z^{*\infty}$ may be used in place of
$C([-2,2])^{*\infty}$ in the preceding definition.
Alternatively, without appealing to this isomorphism, we can deduce from \cite[Proposition 6.3.1]{robertnccw} that $C([-2,2])^{*\infty}$ and
$\mathcal Z^{*\infty}$ embed into each other. This is 
enough to show that either algebra may be used in the definition of
weak selflessness.
\end{remark}

By Theorem \ref{thm:weak-selfless-uniform-asa-inner},
the expectation of a weakly selfless $(A,E,B)$ is a.s.a. inner. 
In the following theorem we give a stronger version 
of this property in this new setting: 

\begin{theorem}\label{thm:prob-dixmier-averaging}
Let $(A,E,B)$ be a weakly selfless C*-probability space. Then there exists a unitary $ w\in B'\cap A^{\cU}$ such that
$$
\Big\| E(a)-\frac1n\sum_{k=1}^n (w^k)^*aw^k\Big\|
\leq
\left(\frac{2}{n}+\frac{2}{\sqrt n}\right)\|a\|
$$
for all $a\in A$ and $n\in \N$.
Consequently, $E(a)\in \overline{\mathrm{co}\{u^*au:u\in B'\cap U(A^{\cU})\}}$ for all $a\in A$.
\end{theorem}

\begin{proof}
We work in $A*_B (B\otimes C(\T))$, and continue to denote the free product expectation by $E$. Let $u\in C(\T)$ be 
a Haar unitary generating $C(\T)$. Let 
$$
A_k=(u^k)^*Au^k\qquad k\in \Z.
$$
We note that, as in the scalar case, the C*-subalgebras $A_k$ are free with amalgamation over $B$, and the restriction of the free product expectation  to $C^*(A_1,A_2,\ldots)$ induces a faithful GNS representation, so that
$$
C^*((A_k)_{k\in \Z})\cong (*_B)_{k\in \Z} A_k.
$$
See the proof of \cite[Lemma 4.1]{DykemaShlyakhtenko2001}.

Now let $a\in A=A_0$, and set $x_k=E(a)-(u^k)^*au^k\in A_k$ for $k\ge0$. Note that $E(x_k)=0$ for all $k$. By Voiculescu's inequality \cite[Proposition 7.4]{junge2005picuineq},
$$
\Big\|\sum_{k=1}^n x_k\Big\|\leq \sup_k \|x_k\|+\Big\|\sum_{k=1}^n E(x_k^*x_k)\Big\|^{\frac12} +\Big\|\sum_{k=1}^n E(x_kx_k^*)\Big\|^{\frac12}.
$$
We have $\|x_k\|\leq 2\|a\|$ and $E(x_k^*x_k)=E(a^*a)-E(a)^*E(a)$, so $\|E(x_k^*x_k)\|\leq \|a\|^2$ and similarly 
$\|E(x_kx_k^*)\|\leq \|a\|^2$. We thus get
$$
\Big\|nE(a)- \sum_{k=1}^n (u^k)^*au^k\Big\|\leq 2\|a\| + 2n^{\frac12}\|a\|. 
$$
Divide by $n$, apply $\sigma\colon A*_B (B\otimes C(\T))\to A^{\cU}$ and set $w=\sigma(u)$ to get the result.
\end{proof}

If $(A,\rho)$ is selfless in the scalar case, the Dixmier property of $\rho$ readily implies  that $A$ is simple and can have at most one trace \cite[Theorem 3.1]{robertselfless}. In the relative  case, we have 
the following: 

\begin{corollary}\label{cor:prob-traces-and-ideals}
If $(A,E,B)$ is weakly selfless, then every bounded trace on $A$ factors through $B$ via $E$ and for each nonzero ideal  $I\subseteq A$, we have that $E(I)=I\cap B\neq 0$.
\end{corollary}	

\begin{proof}
Let $\tau$ be a bounded trace on $A$  and $\tau^{\cU}$ its
ultrapower in $A^{\cU}$. Then $\tau^{\cU}$ remains constant throughout the set $\overline{\mathrm{co}\{uau^*:u\in U(A^\cU)\}}^{\|\cdot \|}$. Thus, by the Dixmier averaging property of $E$, we have $\tau(E(a))=\tau(a)$ for all $a\in A$.

Let $I$ be a nonzero ideal of $A$. By the Dixmier averaging property of $E$, we have $E(I)\subseteq I$, and so
$E(I)=I\cap B$. We cannot have $E(I)=0$, since we have assumed that $E$ induces a  faithful GNS. Thus, $I\cap B\neq 0$.
\end{proof}	

Permanence properties:
\begin{theorem}\label{thm:prob-permanence}
\begin{enumerate}[(i)]
\item
If $(A,E,B)$ is selfless, then so is $(A*_BA', E*_BE',B)$ for an arbitrary $(A',E',B)$.

\item
If $p\in B$ is a projection, then $(pAp, E|_{pAp},pBp)$ is selfless.

\item 
Given $(A,E,B)$ and a conditional expectation $F\colon B\to C$, if $(A,FE,C)$ is selfless and 
 $F$ has finite Watatani index, then $(A,E,B)$ is selfless.

\item
If $(A,E,B)$ is selfless and $F\colon B\to C$ is an  a.s.a. inner expectation, then $(A,FE,C)$ is selfless.

\item 
If $(A_1,E_1,B_1)$ is selfless, $A_2$ is exact, and $E_2\colon A_2\to B_2$ is a.s.a. inner, then $(A_1\otimes  A_2,E_1\otimes E_2,B_1\otimes B_2)$ is selfless.	 
\end{enumerate}
\end{theorem}

\begin{proof}
(i) This follows from ``free exactness".

(ii) This follows from the corners version of Blanchard-Dykema.	

(iii) By finite Watatani index we mean the existence of a quasibasis $\{x_i,y_i:i=1,\ldots,n\}$ in $B$ such that
$$
b=\sum_{i=1}^n y_iF(x_ib)=\sum_{i=1}^n F(by_i)x_i
$$
for all $b\in B$. We note that the quasibasis can always be chosen of the form $\{x_i,x_i^*:i=1,\ldots, n\}$ \cite[Lemma 2.1.6]{Watatani1990}. Let us assume that it has this form. The element
$e=\sum_{i=1}^n x_ix_i^*$ is the a central invertible element in $C$  called the index of $F$. Set $x_i' =e^{-1/4}x_i$ for $i=1,\ldots,n$.
Then it is not hard to check that 
$$
b=\sum_{i,j} x_i' F((x_i')^*bx_j')(x_j')^*, 
$$
for all $b\in B$ (see \cite[Example 5.2.2]{Hasegawa2015RelativeNuclearity}).
Then
\[
E(a)=\sum_{i,j} x_i' FE((x_i')^*ax_j')(x_j')^*. 
\]
for all $a\in A$. Since $FE$ is selfless, it follows by Theorem \ref{thm:cp-map-permanence} that $E$ is selfless.

(iv) Since $E$ is weakly selfless, Theorem
\ref{thm:weak-selfless-uniform-asa-inner} implies that $E$ is a.s.a. inner; hence $FE$ is a.s.a. inner.
 We can thus apply Theorem \ref{thm:cp-map-permanence} with  $\phi=E$ and $\psi=FE$. 

(v) This follows from Theorem \ref{thm:cp-map-permanence}.
\end{proof}

\begin{corollary}\label{cor:tensor-with-exact-monotracial}
If $(A_1,\rho_1)$ is selfless and $(A_2,\rho_2)$ is exact, simple, with unique trace $\rho_2$, then $(A_1\otimes A_2, \rho_1\otimes \rho_2)$
is selfless.
\end{corollary}
\begin{proof}
If $A_2$ is simple with unique tracial state $\rho_2$, then $A_2$ has the
Dixmier property by the theorem of Haagerup and Zsidó
\cite{HaagerupZsido1984}. It follows that 
$x\mapsto \rho_2(x)1_{A_2}$ is approximately inner. It is easy to see from this that $\rho_2 1_{A_2}$ is a.s.a. inner (using that it ranges $\C 1$).
We can thus apply Theorem \ref{thm:prob-permanence}.
\end{proof}

%\todo{explain}
%$$
%\frac12(waw^* + w^*aw)=\frac{w+w^*}{2}a\frac{w+w^*}{2} + \frac{w-w^*}{2i}a\frac{w-w^*}{2i}.
%$$

\begin{theorem}\label{thm:finite-dimensional-base-classification}
Let $(A,E,B)$  be a C*-probability space with   $B$ finite dimensional and  $E$ faithful. Then $(A,E,B)$ is selfless if and only if 
$$
(A,E,B)\cong (A_1,E_1,B_1)\oplus \cdots \oplus (A_m,E_m,B_m)
$$
where each $A_i$ is selfless in the scalar sense relative to a state $\rho_i$, and if $\rho_i$ is tracial then it is preserved by $E_i$.
\end{theorem}
%\todo{add that in this case weakly selfless and selfless agree, by uniqueness of trace/purely infinite dichotomy}
\begin{proof}
Suppose that $(A,E,B)$ is selfless.
Let us  first sort out the direct sum decomposition of $A$. 
Let $I$ be an ideal of $A$. By Corollary \ref{cor:prob-traces-and-ideals}, we have that $E(I)=I\cap B$ is an ideal of $B$. Hence, it has the form $e_IB$ for some central projection $e_I\in B$. 
If $a\in I$, then
$$
E((1-e_I)aa^*(1-e_I))=(1-e_I)E(aa^*)(1-e_I)=0.
$$
By the faithfulness of $E$ we deduce that $e_Ia=a$. Thus,  $I=e_IA$. We deduce from this that $A$ decomposes into a direct sum of simple
C*-subalgebras $A_i=e_iA$, for $i=1,\ldots,m$, where $e_1,\ldots,e_m$ are pairwise orthogonal central projections of $B$. We thus obtain the direct sum decomposition of the theorem with $E_i=E|_{A_i}$ and $B_i=e_iB$.
Note that, by the preservation of selflessness when passing to corners of $B$ (Theorem \ref{thm:prob-permanence}), each $(A_i,E_i,B_i)$ is selfless.

Let us sort out next the structure of the simple summands.
Fix $i$, and pick a minimal projection $p\in B_i$. This way $pB_ip\cong \C$. By Theorem \ref{thm:prob-permanence},   $(pA_ip,E|_{pA_ip},\C)$ is selfless. That is, $pA_ip$ is selfless in the scalar sense relative to the state $\rho_{p,i}=E|_{pA_ip}$. Since $A_i$ is simple, the projection $p$ generates $A_i$ as an ideal. Thus, by Brown's theorem, $A_i$ is isomorphic to a corner  of a matrix algebra: $A_i\cong qM_{N}(pA_ip)q$. Since selflessness is preserved under matrix amplifications and corners (\cite{robertselfless}, \cite{gould2026selfless}), $A_i$ is selfless. If $A_i$ is tracial, then its trace must be preserved by $E_i$, by Corollary \ref{cor:prob-traces-and-ideals}.

Suppose conversely that $(A,E,B)$ admits the direct sum decomposition of the theorem. Given C*-probability spaces
$(A,E,B)$, $(C,F,B)$, $(A',E',B')$, $(C',F',B')$, we have 
a canonical isomorphism
$$
(A\oplus A')*_{B\oplus B'} (C\oplus C')\cong (A*_BC)\oplus (A'*_{B'}C').
$$
From this we deduce that it suffices to show that each summand is selfless.

Assume that $A$ is selfless in the scalar sense, that $E\colon A\to B$ is a faithful conditional expectation onto a finite dimensional subalgebra, and that if $A$ is tracial, then $E$ preserves the trace. If $A$ is not tracial, then it is purely infinite, and so it  is selfless relative to any state. In this case, choose a faithful trace $\tau\colon B\to \C$ and set $\rho = \tau \circ E$. Then  in both the tracial or the non-tracial case, we can arrange that $(A,\rho)$ is selfless where $\rho$ is faithful, preserved by $E$, and inducing a trace on $B$.  Then $(B,\rho|_B)$ has finite Watatani index, as $B$
is finite dimensional \cite{Watatani1990}. Thus, $(A,E,B)$ is selfless by Theorem \ref{thm:prob-permanence}.
\end{proof}

\section{Interactions of selflessness with nuclearity}\label{sec: nuclearity}

The definition of relative nuclearity below appears in Hasegawa \cite{Hasegawa2015RelativeNuclearity} (applied to conditional expectations).

\begin{definition}\label{def:relative-nuclearity}
Let $\phi\in CP(A)$. Call $\phi$ real relatively nuclear if
$\mathrm{id}_A$ is a point-norm limit of cp maps of the form
$$
a\mapsto
\sum_{k=1}^m \sum_{i,j=1}^{n_k}
x_{k,i}\phi(x_{k,i}^*ax_{k,j})x_{k,j}^*,
\qquad a\in A,
$$
where $x_{k,i}\in A$. Call $\phi$ relatively nuclear if
the identity $\mathrm{id}_A$ is a point-norm limit of cp maps  of the form
$$
a\mapsto
\sum_{k=1}^m \sum_{i,j=1}^{n_k}
y_{k,i}\phi(x_{k,i}^*ax_{k,j})y_{k,j}^*,\qquad a\in A,
$$
where 	$x_{k,i},y_{k,i}\in A$.
\end{definition}
%some loss of generality by not including real coefficients t_{k,i}?

\begin{remark}
It is not difficult to show that $\phi$ is relatively nuclear if and only if we have the weak containment of C*-correspondences  $A\precsim A\otimes_\phi A$, and it is real relatively nuclear if  and only if we have the weak containment of C*-correspondences with real structures 
$(A, \std)\precsim (A\otimes_\phi A, \std)$.
\end{remark}

\begin{theorem}\label{thm:relnuc}\label{thm:relative-nuclearity-to-weak-selflessness}
Let $E\colon A\to B$ be a unital conditional expectation. Let $\psi\in CP(B)$.
\begin{enumerate}[(i)]
\item
If $(A,\psi\circ E)$ is weakly selfless and $\psi$ is real relatively nuclear, then $(A,E,B)$ is weakly selfless.
\item
If $(A,\psi\circ E)$ is weakly Toeplitz selfless and $\psi$ is relatively nuclear, then $(A,E,B)$ is weakly Toeplitz selfless.
\end{enumerate}
\end{theorem}

\begin{proof}
(i) Choose a net of maps $(\psi_\lambda)_{\lambda}\in CP(B)$
as in the definition of real relative nuclearity applied to  $\psi$,
converging to the identity of $B$ in the point-norm topology. 
From the fact that $\psi\circ E$ is weakly selfless, we deduce that
$\psi_{\lambda}\circ E$ is weakly selfless by
Theorem \ref{thm:cp-map-permanence} (i).
Hence $E=\lim_{\lambda} \psi_{\lambda}\circ E$ is weakly selfless by
Theorem \ref{thm:cp-map-permanence}.

(ii) We argue as in (i), using the Toeplitz version of the cp-map permanence theorem. The more general form of the cp maps appearing in the definition of relative
nuclearity is handled by
Theorem \ref{thm:toeplitz-cp-map-permanence}. Thus
$\psi_\lambda\circ E$ is weakly Toeplitz selfless for each $\lambda$, and the point-norm limit argument gives weak Toeplitz selflessness.
\end{proof}

\begin{corollary}\label{cor:relnuc-selfless-absorbs-Z-Oinfty}
Let $A$ be a separable unital C*-algebra and let $\phi\in CP(A)$.
\begin{enumerate}[(i)]
\item
If $\phi$ is weakly selfless and real relatively nuclear, then
$A\otimes\mathcal Z\cong A$.

\item
If $\phi$ is weakly Toeplitz selfless and relatively nuclear, then
$A\otimes\mathcal O_\infty\cong A$.
\end{enumerate}
\end{corollary}

\begin{proof}
For part (i), apply Theorem
\ref{thm:relative-nuclearity-to-weak-selflessness}
with $B=A$, $E=\id_A$, and $\psi=\phi$. Since
$\phi=\phi\circ\id_A$ is weakly selfless and $\phi$ is real relatively
nuclear, the theorem shows that $(A,\id_A,A)$ is weakly selfless.
Consequently, there is a *-homomorphism
$$
A\otimes\mathcal Z^{*\infty}
\cong
A*_A(A\otimes\mathcal Z^{*\infty})
\longrightarrow A^{\cU}
$$
whose restriction to $A$ is the diagonal embedding. Restricting this
homomorphism to $1_A\otimes\mathcal Z$ gives a unital embedding
$\mathcal Z\to A'\cap A^{\cU}$.
Since $A$ is separable,
$A\otimes\mathcal Z\cong A$ \cite[Theorem 2.2]{TomsWinter2007}.

To prove (ii), apply   Theorem
\ref{thm:relative-nuclearity-to-weak-selflessness} (ii) with
$B=A$, $E=\id_A$, and $\psi=\phi$. This yields  a unital embedding
$\mathcal O_\infty\to A'\cap A^{\cU}$.
Since $A$ is separable, $A\otimes\mathcal O_\infty\cong A$ \cite{RordamStormer-EMS126}.
\end{proof}

\begin{remark}
Expanding the  proof of (i) some more, we can show this:  if $A$ is exact, then $(A,\phi)$
is both selfless and real relatively nuclear if and only if $\phi$ is a.s.a. inner 
 and  the first factor embedding $(A,\phi)\to (A\otimes \mathcal Z,\phi\otimes\tau)$ is existential. For separable C*-algebras, a C*-algebraic existential embedding into $A\otimes \mathcal Z$ is equivalent to $\mathcal Z$-stability, but the standard proof of this fact does not directly produce an isomorphism 
  $(A,\phi)\cong (A\otimes \mathcal Z,\phi\otimes\tau)$
assuming  that the first factor embedding
 $(A,\phi)\to (A\otimes \mathcal Z,\phi\otimes\tau)$ is existential and $A$ separable; we instead get a sequence of C*-algebraic isomorphisms that asymptotically intertwine $\phi$ and $\phi\otimes\tau$.
\end{remark}

\begin{corollary}\label{cor:kirchberg-oinfty-stability}
[Kirchberg's $\mathcal O_\infty$-stability theorem]
If $A$ is separable, purely infinite, simple, unital, and nuclear, then $A\cong A\otimes\mathcal O_\infty$.
\end{corollary}

\begin{proof}
Choose any state $\rho$ on $A$. By Ozawa's theorem (or rather, its proof), $(A,\rho)$ is Toeplitz selfless \cite[Theorem 3]{ozawa2025proximalityselflessnessgroupcalgebras}. 
Since $A$ is nuclear and $\rho$ induces a faithful GNS representation, $(A,\rho,\C)$ is relatively nuclear \cite[Example 5.1.1]{Hasegawa2015RelativeNuclearity}.
By the previous corollary,  $A\cong A\otimes\mathcal O_\infty$.
\end{proof}

 A tracial C*-probability space 
$(A,\rho)$ is called MF if it admits a trace preserving embedding into 
$(\mathcal Q^{\cU},\tau^{\cU})$, where $\mathcal Q=\bigotimes_{n=1}^\infty M_n(\C)$. In the theorem below we apply Theorem
\ref{thm:relative-nuclearity-to-weak-selflessness}
to show that certain reduced amalgamated free products of MF C*-probability spaces are again MF.

\begin{theorem}\label{thm:mf-amalgamated-free-product}
Let $(A,\rho)$ be an MF C*-probability space. Let $E\colon A\to B$ be trace preserving 
expectation onto a real nuclear C*-subalgebra $B$.
Then 
$$
(A*_B A, \rho \circ (E*_BE))
$$
is MF.
\end{theorem}

\begin{proof}
We first reduce the proof to the case that $(A,\rho)$ is selfless. 
Consider 
$$
(A',\rho')=(A,\rho)*(\mathcal Q,\tau).
$$
Then $(A',\rho')$
is selfless and MF \cite{robertselfless}. The expectation   $E'=E*\tau\colon A'\to B$
extends $E$ and preserves $\rho'$. By Blanchard-Dykema, $A*_BA$ embeds in an expectation preserving way into $ A'*_B A'$, so it suffices to show that $A'*_B A'$ is MF. Renaming $(A',\rho')$ as  $(A,\rho)$, let us henceforth assume that $(A,\rho)$ is selfless.

Applying Theorem \ref{thm:relative-nuclearity-to-weak-selflessness}  with $\psi=\rho|_B$, we deduce that $(A,E,B)$ is weakly selfless.
In particular, $C=A*_B (B\otimes C_r^*(\mathbb F_\infty))$ embeds in $A^{\cU}$. To see that this embedding is trace preserving, it will suffice to show that  
$C$ has a unique trace. Let $F\colon C\to B$ be the amalgamated free product conditional expectation onto $B$. Notice that $(C,F,B)$ is selfless, by Theorem \ref{thm:prob-permanence} (i), (v).
Thus, every trace on $C$ factors through $F$.
But if $\theta$ is a tracial state on $C$, then its restriction to
$A$ must agree with $\rho$, by selflessness of $(A,\rho)$. Hence,
$$
\theta=\rho|_B\circ F.
$$
This shows that $C$ has a unique trace, whence its embedding in $A^{\cU}$ is trace preserving.

We obtain a trace-preserving embedding of $A*_B A$ into $C$ by
choosing a Haar unitary $u\in C_r^*(\mathbb F_\infty)$ and mapping
the first copy of $A$ identically and the second copy to
$(1\otimes u)A(1\otimes u)^*$.
Thus,  $A*_B A$ embeds in $A^{\cU}$ in a trace preserving way. Since $A$ is MF, we conclude that $A*_B A$ is MF.
\end{proof}

A discrete group $G$ is called MF if $(C_r^*(G),\tau)$ is MF, where $\tau$ is the canonical trace.

\begin{corollary}\label{cor:mf-amalgamated-free-product-groups}
Let $G$ be an MF group with amenable subgroup $H$. Then $G*_H G$ is MF.
\end{corollary}

\begin{proof}
We have 
$$
C_r^*(G*_H G)\cong C_r^*(G)*_{C_r^*(H)} C_r^*(G)
$$
\cite[Chapter 4]{voiculescu1992free}. Moreover, 
 if $H\subseteq G$ is amenable, then $C_r^*(H)$ is real nuclear \cite[Theorem 2.6.8]{brown2008textrm}.
We can thus apply the previous theorem.
\end{proof}

\section{The PHP property and crossed products}\label{sec: crossed products}

In \cite{ozawa2025proximalityselflessnessgroupcalgebras} Ozawa introduced the PHP property for groups. 
We recall its definition here (stated slightly differently):

\begin{definition}\label{def:php-property}
The discrete group $G$ has the PHP property if, for every finite set
$\Lambda\subseteq G\setminus\{e\}$, every $\delta>0$, and every
$N\in\N$, there exist $n\geq N$ and triples
$(t_i,\Gamma_i,\Gamma_i^+)$, $i=1,\ldots,n$, with
$t_i\in G$ and $\Gamma_i^+\subseteq\Gamma_i\subseteq G$, satisfying the following
conditions:
\begin{enumerate}
\item[P1:]	
The sets 
$$
\Gamma_i^+, \, t_j^{-1}(G\setminus \Gamma_j),\quad i,j=1,\ldots,n
$$
are pairwise disjoint.
\item[P2:] 
With
$$
\Gamma=\bigcup_i \Gamma_i^+\cup  \bigcup_i t_i^{-1}(G\setminus \Gamma_i),
$$
we have $g\Gamma \cap \Gamma=\varnothing$, for all $g\in \Lambda$, 
$e\notin \Gamma$, and $g^{-1}\notin \Gamma$ for all $g\in \Lambda$.

\item[P3:]
For each $g\in G$, the cardinalities of the sets
$$
\{i:g\in\Gamma_i\setminus\Gamma_i^+\},
\qquad
\{i:g\in t_i^{-1}(\Gamma_i\setminus\Gamma_i^+)\}
$$
are at most $\delta\sqrt n$.
\end{enumerate}
\end{definition}

By \cite[Theorem 14]{ozawa2025proximalityselflessnessgroupcalgebras},
if $G$ has the PHP property then $C_r^*(G)$ is completely selfless, whence selfless. The class of groups that  have the PHP property is closed under direct products \cite[Section 8]{ozawa2025proximalityselflessnessgroupcalgebras} and contains all acylindrically hyperbolic groups with trivial finite radical \cite[Proposition 15]{ozawa2025proximalityselflessnessgroupcalgebras}. 

Let $G$ be a discrete group acting on a unital C*-algebra $B$: $g\mapsto \alpha_g\in \mathrm{Aut}(B)$. We recall  the construction of the reduced crossed product $B\rtimes_r G$ and of the expectation $E\colon B\rtimes_r G\to B$, and along the way introduce the notation that we will use below.

Assume that $B$ is embedded unitally  in $B(H)$, for a Hilbert space $H$. Let $\ell^2(G,H)$ be the Hilbert space of square summable functions 
$\xi\colon G\to H$. Let $B$ and $G$ act on $\ell^2(G,H)$ as follows:
\[
(b\xi)(g)=\alpha_{g^{-1}}(b)\xi(g),\qquad (u_h\xi)(g)=\xi(h^{-1}g),
\]
for $b\in B$, $g,h\in G$, and $\xi\in \ell^2(G,H)$. Then $B\rtimes_r G=C^*(B,u_g:g\in G)$.
We note that $u_gb=\alpha_g(b)u_g$, from which it follows that  finite sums of the form
\[
\sum_{g} b_gu_g
\]
form a  dense *-subalgebra of 	$B\rtimes_r G$.

Let $V_e\colon H\to \ell^2(G,H)$ be the embedding of $H$ into the summand at $g=e$, i.e., $V_e(v)=v\delta_{e}$ for $v\in H$. 
Note that $V_eb=bV_e$ for $b\in B$ (where $b$ acts on $H$ on the left-hand side and on $\ell^2(G,H)$ on the right-hand side). 
Define a u.c.p. map $E\colon B(\ell^2(G,H))\to B(H)$ by  $E(x)=V_e^*xV_e$. Then $E$ is a $B$-bimodule map and
\[
E(\sum_g b_gu_g)=b_e.
\]
So $E$ defines an expectation from $B\rtimes_r G$ onto $B$.  

Regard $\ell^\infty(G)\subset B(\ell^2(G,H))$ as multiplication operators.  Set
$$
A=B\rtimes_r G,
\qquad
C=C^*(A,\ell^\infty(G))
\subseteq B(\ell^2(G,H)).
$$
The finite linear span of the elements $bfu_g$, with
$b\in B$, $f\in\ell^\infty(G)$, and $g\in G$, is a dense
$*$-subalgebra of $C$. Since $E(bfu_g)=0$ if $g\neq e$, and $E(bf)=f(e)b$,
the  map $E$ restricts to a $B$-valued conditional expectation on $C$.

\begin{theorem}\label{thm:php-crossed-product-selfless}
If $G$ is a discrete group with the PHP property acting on $B$ by  approximately inner automorphisms,
then $(B\rtimes_{r} G,E,B)$ is selfless.
\end{theorem}		

The proof adapts Ozawa's PHP argument for reduced group C*-algebras
\cite[Theorem 14]{ozawa2025proximalityselflessnessgroupcalgebras} to the relative setting, with
approximate innerness used to absorb the coefficients from $B$.

\begin{proof}
Set $A=B\rtimes_r G$ and $C=C^*(A,\ell^\infty(G))$,  and regard $E$ as a $B$-valued conditional expectation on $C$.  
We will construct for each finite set $F\subseteq A$ and $\epsilon>0$, an operator $T\in C$ such that 
\begin{enumerate}
\item
$T^*aT\approx_\epsilon E(a)$ for all $a\in F$,
\item
$E(aTT^*a^*)\approx_\epsilon 0$ for all $a\in F$, 
\item 
$T+T^*\approx_{\epsilon} s$ for some $s\in A$.
\end{enumerate}
By Proposition \ref{prop:cp-map-selflessness-criterion-epsilon}, this implies that $E$ is selfless as a cp map on $A$, i.e., $(A,E,B)$ is selfless (Theorem \ref{thm:prob-selfless-equivalent-cp-map}).

Let $F\subset A$ be finite and $\epsilon>0$. To construct $T$, it suffices to choose $F$ from a set with dense linear span. So let us assume that 
$$
F=\{b_gu_g : g\in \Lambda\}\cup S,
$$
where $\Lambda\subseteq G\setminus\{e\}$, $b_g\in B$, and $1\in S\subset B$.

Let $\delta>0$ be  small (how small to be made precise below). By the PHP property of $G$ applied to the finite set $\Lambda$ and to $\delta$, there exist $n$ and triples $(t_i,\Gamma_i,\Gamma_i^+)$, $i=1,2,\ldots,n$, 
with $t_i\in G$ and $\Gamma_i^+\subseteq \Gamma_i\subseteq G$, 
satisfying P1--P3 from Definition \ref{def:php-property} and  $\|b\|/\sqrt n<\delta $ for all $b\in S$. 

From the triples $(t_i,\Gamma_i,\Gamma_i^+)$ define triples $(u_i,P_i,P_i^+)$,
where $u_i=u_{t_i}\in A$ is a unitary, and $P_i,P_i^+\in \ell^\infty(G)\subseteq B(\ell^2(G))$ are
the projections
$$
P_i=P_{\Gamma_i},\quad P_i^+=P_{\Gamma_i^+}.
$$
We regard $u_i,P_i,P_i^+$ as operators in $C$.
Note that the projections $P_i,P_i^+$ commute with $B$, since 
elements of $B$ act fiberwise on vectors in $\ell^2(G,H)$. Set
$$
P=\sum_{i=1}^n P_i^+ + \sum_{i=1}^n u_i^*(1-P_i)u_i.
$$
From the properties of the triples $(t_i,\Gamma_i,\Gamma_i^+)$ we obtain that
\begin{itemize}
\item
the $2n$ projections $P_i^+,u_i^*(1-P_i)u_i$, $i=1,\ldots,n$,
are pairwise orthogonal (whence, $P$ is a projection);

\item
$P u_g P=0$ for all $g\in \Lambda$,  $PV_e=0$, and $Pu_{g^{-1}}V_e=0$ for all  $g\in \Lambda$.

\item
$$
\Big\|\sum_{i=1}^n (P_i-P_i^+)\Big\|\le \delta \sqrt n,\qquad
\Big\|\sum_{i=1}^n u_i^*(P_i-P_i^+)u_i\Big\|\le \delta\sqrt n.
$$
\end{itemize}

Since $\alpha_{t_i^{-1}}$ is approximately inner, 
we may choose for each $i=1,\ldots,n$ a unitary  $w_i\in U(B)$ such that
\[
\|w_i b w_i^*-\alpha_{t_i^{-1}}(b)\|<\delta,
\]
for all $b\in S$.

Define
\[
T =\frac1{\sqrt{2n}}\sum_{i=1}^n (P_i^+u_i w_i + w_i^*u_i^*(1-P_i))\in C.
\]
The $2n$ summands in the definition of $T$ have pairwise orthogonal
ranges: the range of $P_i^+u_iw_i$ is contained in $P_i^+$, while
the range of $w_i^*u_i^*(1-P_i)$ is contained in
$u_i^*(1-P_i)u_i$. Hence the cross terms in $T^*T$ vanish, and
\[
T^*T = \frac1{2n}\sum_{i=1}^n \Big(w_i^*u_i^*P_i^+u_iw_i + (1-P_i)u_iw_iw_i^*u_i^*(1-P_i)
\Big) \leq 1.
\]
In particular, $\|T\|\leq1$.

Let us show that  $T$ satisfies the three properties (1)-(3) listed at the start of the proof for a sufficiently small $\delta$.

\emph{Proof of (1)}. 
Let $x=b_g u_g$, with $g\in \Lambda$. As remarked above,   $P_i$ and $P_i^+$ commute with $b_g\in B$ for all $i$.
Thus
\[
PxP=Pb_gu_gP=b_gPu_gP=0.
\]
The range of $P_i^+u_iw_i$ is contained in the range of $P_i^+$, whereas the range of $w_i^*u_i^*(1-P_i)$ is contained in the range of $u_i^*(1-P_i)u_i$. We deduce from this that  $T=PT$. Hence,
\[
T^*xT=T^*P x P T=0.
\]

 Now let $b\in S$. A direct computation shows that
\[
T^*bT = \frac1{2n}\sum_{i=1}^n
\Big(w_i^*\alpha_{t_i^{-1}}(b)w_i\,u_i^*P_i^+u_i
+\alpha_{t_i}(w_i b w_i^*)(1-P_i)\Big).
\]
By our choice of the unitaries $w_i$, 
$
\|w_i^*\alpha_{t_i^{-1}}(b)w_i - b\|<\delta$ and  $\|\alpha_{t_i}(w_i b w_i^*) - b\|<\delta$.
Applying these estimates on the right hand side of the expression for $T^*bT$, we get
\[
\Big\|T^*bT -\frac1{2n}\sum_{i=1}^n
(bu_i^*P_i^+u_i+b(1-P_i))\Big\|<\delta.
\]
Thus
\begin{align*}
\|T^*bT-b\| & \leq \delta+\frac{\|b\|}{2n}\Big\|\sum_i (u_i^*(1-P_i^+)u_i + P_i)\Big\|\\
&\leq \delta+\frac{\|b\|}{2n}\|P\| + \frac{\|b\|}{2n}\Big\|\sum_i u_i^*(P_i-P_i^+)u_i\Big\|
+\frac{\|b\|}{2n}\Big\|\sum_i (P_i-P_i^+)\Big\|\\
&\leq 3\delta.
\end{align*}
Choosing $\delta$ sufficiently small and $n$ sufficiently large gives (1).

\emph{Proof of (2)}. Since $PV_e=0$, we have $E(P)=V_e^* P V_e =0$. Let $b\in B$. Since $T$ is a contraction
and $PT=T$,  we have  $TT^*\le P$. Hence,
\[
0\le E(bTT^*b^*)\le bE(P)b^*=0.
\]
Now let $x=b_g u_g$, with $g\in \Lambda$. 
Since $TT^*\le P$, 
\[
xTT^*x^*\leq b_g u_gPu_{g^{-1}}b_g^*.
\]
Applying $E$ we get
\[
0\le E(xTT^*x^*)\le b_gE(u_gPu_{g^{-1}})b_g^*.
\]
Since $Pu_{g^{-1}}V_e=0$, 
\[
E(u_gPu_{g^{-1}})=V_e^* u_g P u_{g^{-1}} V_e=0.
\]
Thus, $E(xTT^*x^*)=0$. This proves (2).

\emph{Proof of (3)}. Set
\[
s=\frac1{\sqrt{2n}}
\sum_{i=1}^n (u_i w_i+w_i^*u_i^*)\in B\rtimes_r G.
\]
Since $w_i$ commutes with $P_i$ and $P_i^+$,
\[
T+T^*-s
=-\frac1{\sqrt{2n}}\sum_{i=1}^n
((P_i-P_i^+)u_i w_i+w_i^*u_i^*(P_i-P_i^+)).
\]
By the operator Cauchy-Schwarz inequality,
\[
\Big\|\sum_{i=1}^n (P_i-P_i^+)u_i w_i\Big\|
\leq
\Big\|\sum_{i=1}^n (P_i-P_i^+)\Big\|^{\frac12}
\Big\|\sum_{i=1}^n u_i^*(P_i-P_i^+)u_i\Big\|^{\frac12}
\leq \delta\sqrt n.
\]
Hence
\[
\|T+T^*-s\|\leq \sqrt 2\delta.
\]
Thus, choosing $\sqrt 2\delta <\epsilon$  proves (3).
\end{proof}

\bibliography{references}
\bibliographystyle{plain}
\end{document}